\documentclass[12pt]{amsart}

\usepackage{amsfonts}
\usepackage{mathrsfs}
\usepackage{amsmath}
\usepackage{amssymb}
\usepackage{amsthm}
\usepackage{enumitem}
\usepackage{latexsym}
\usepackage{color}
\usepackage{geometry}
\theoremstyle{definition}
\newtheorem{defin}{Definition}[section]
\newtheorem{prop}[defin]{Proposition}
\newtheorem{lem}[defin]{Lemma}

\newtheorem{teo}[defin]{Theorem}

\newtheorem{question}[defin]{Question}
\newtheorem*{claim}{Claim}

\newtheoremstyle{TheoremNum}
{\topsep}{\topsep}              
{\itshape}                      
{}                              
{\bfseries}                     
{.}                             
{ }                             
{\thmname{#1}\thmnote{ \bfseries #3}}
\theoremstyle{TheoremNum}

\DeclareMathOperator{\supp}{ \text{supp }}

\newcommand{\comment}[1]{}
\begin{document}

	\title[Comfort's question and Wallace semigroups]{Comfort's question on powers in $\mathbb Q ^{(2^\mathfrak c)}$ and a Wallace semigroup whose cube is countably compact}
	\author[J. L. J. Fuentes-Maguiña]{Juan Luis Jaisuño Fuentes-Magui\~na}
	\author[A. H. Tomita]{Artur Hideyuki Tomita}
	\address{Depto of Matem\'atica, Instituto of Matem\'atica and Estat\'istica, Universidade of S\~ao Paulo, Rua do Mat\~ao, 1010 -- CEP 05508-090, São Paulo, SP - Brazil}
	\email{tomita@ime.usp.br, juanfm@ime.usp.br}
	\thanks{The first listed author has received financial support from CAPES (Brazil) as a PhD student at the University of São Paulo under supervision of the second listed author.}
	\thanks{The second listed author has received financial support from FAPESP 2021/00177-4.}
	\date{}
	\commby{}

	\begin{abstract}
		We prove that the existence of $\mathfrak{c}$ incomparable selective ultrafilters implies the existence of a Wallace semigroup whose cube is countably compact. In addition, assuming the existence of $2^{\mathfrak c}$ incomparable selective ultrafilters and $2^{< 2^{\mathfrak{c}}} = 2^{\mathfrak{c}}$, we obtain torsion-free topological groups with respect to Comfort's question on the countable compactness of (infinite) powers of a topological group.
	\end{abstract}
	
	\maketitle

\section{Introduction}

    \subsection{Some history}

    In the Open Problems in Topology, W. W. Comfort asked in 1990 the following question: Is there, for every cardinal number $\alpha \leq 2^{\mathfrak{c}}$, a topological group $G$ such that $G^{\gamma}$ is countably compact for all cardinals $\gamma < \alpha$, but $G^{\alpha}$ is not countably compact?

    Tomita \cite{Tomita05} showed in 2005 that it is consistently affirmative for every cardinal $\kappa \leq 2^{\mathfrak c}$. These examples are groups of order 2 and were built from the existence of selective ultrafilters.

    In 2019, Tomita \cite{T19} proved that there exists a consistent torsion-free group that answer this question affirmatively for each $\alpha \leq \omega$. The example is in fact a free Abelian group, but as shown in \cite{T98}, the countable power of a topological free Abelian group is not countably compact. Thus, some different torsion-free group should be used for larger $\alpha$.

    We construct an example, for every infinite cardinal $\alpha \leq 2^\mathfrak c$, using an appropriate group topology on $\mathbb Q ^{(2^\mathfrak c)}$. The construction rely on the technique developed in Bellini, Rodrigues and Tomita \cite{Tomita21}, who constructed a $p$-compact group topology on $\mathbb Q^{(\kappa)}$ for every infinite cardinal $\kappa = \kappa ^\omega$ and a selective ultrafilter $p$, together with some combinatorial properties of many incomparable selective ultrafilters as it was done in \cite{Tomita05}.

    Wallace \cite{W} asked whether a both-sided cancellative countably compact semigroup is a topological group. This problem divides in two: the conditions that a semigroup must be algebraically a group and the conditions that make semigroup topology on a group have a group topology. It was shown that pseudocompact semigroup topologies in a group are group topologies \cite{R}. 

    Robbie and Svetlichny \cite{RS} were the first to show that there exists a both-sided cancellative semigroup which is not a group. Such semigroups have been called Wallace semigroups since then. Other examples were obtained using Martin's Axiom for countable poses, $\mathfrak c$ selective ultrafilters and a single selective ultrafilter. 

    On the other hand, Grant \cite{Grant} showed that there  no Wallace semigroups that are sequentially compact. This motivated him to ask the following questions about other properties related to the power of the semigroup:
    \begin{enumerate}
        \item[(1)] Is every $p-$compact cancellative topological semigroup a topological group? 
        \item[(2)] Is $S$ a topological group if $S \times S$ is countably compact? 
    \end{enumerate}

    Tomita \cite{T96} showed that there are no $p$-compact Wallace semigroups (that is, if $S$ is a Wallace semigroup then $S^{2^\mathfrak c}$ is not countably compact). This answers $(1)$. 

    For the second question, Boero and Tomita \cite{Boero} gave a consistent example of a Wallace semigroup whose square is countably compact. The example was a semigroup of a free Abelian group whose square is countably compact. 

    In \cite{T15} it was shown that it is consistent that there exists a free Abelian group whose every finite power is countably compact, but it was not possible to obtain from it a Wallace semigroup whose finite powers are countably compact.

    In this paper, we obtain a Wallace semigroup whose cube is countably compact making use of the fact that the example is a semigroup of a $\mathbb Q$-vector space rather than a free Abelian group. 
    
    \subsection{Basic notation and terminology} In what follows, $\mathbb{Q}_+$ will denote  the subset of non-negative rationals, $\mathbb N$ the subset of positive integers. Let $\mathbb T$ be the Abelian topological group $\mathbb R/\mathbb Z$ with a metric $d$ given by $d([x],[y]) = \min \lbrace \vert x-y + a \vert : a \in \mathbb{Z} \rbrace$, for each $x, y \in \mathbb{R}$.
    
    Let $X$ be a non-empty subset and $G$ be a group with neutral element 0. The support of a function $f : X \to G$ is defined by $\supp{f} = \lbrace x \in X \mid f(x) \neq 0 \rbrace$. The group $\lbrace f \in G^X : \vert \supp{f} \vert < \omega \rbrace$ will be denoted by  $G^{(X)}$. In addition, if $Y \subset X$, $G^{(Y)}$ will be the subgroup $\lbrace f \in G^{(X)} \mid \supp{f} \subset Y \rbrace$.
    
    For any $a \in X$, we set $\vec {a}$ as the constant with value $a$, $\chi_{a}$ as the characteristic function with $\supp \chi_{a} = \{ a \}$ and $\chi_{a}(a)=1$, and $\chi_{\vec{a}}$ as the constant sequence with value $\chi_{a}$. 


	The set of all free ultrafilters on $\omega$ is denoted by $\omega^*$. Let $p \in \omega^*$, we define a equivalence relation on $X^\omega$ by letting $f\simeq_p g$ if and only if $\{n \in \omega \mid f(n)=g(n)\} \in p$. Given $f \in X^{\omega}$, let $[f]_p$ be the equivalence class determined by $f$. The ultrapower of $X$ is the set of all $p-$equivalence classes in $X^{\omega}$ and will be denoted by $Ult_p(X)$. 
	
	Throughout this paper, we fix an infinite cardinal $\kappa$ such that $\kappa^\omega = \kappa$. We will work with ultrapowers of $\mathbb Q^{(\kappa)}$ taking advantage of its natural $\mathbb Q$-vector space structure. 

\section{Countable compactness, $p-$compactness and selective ultrafilters}

 We assume that all topological spaces are Hausdorff.
    In this section, we review some basic facts about  countable compactness,    $p-$compactness and selective ultrafilters.
    
    \begin{defin}
    A topological space $X$ is countably compact if every countable open cover has a finite subcover.
    \end{defin}
    
    It is not difficult to show that a topological space $X$ is countably compact if, and only if, every sequence in $X$ has an accumulation point. 
    
    Furthermore, as each accumulation point of a sequence is also a $p-$limit of that sequence, for some ultrafilter $p$, we consider the following class of spaces.   
    
    \begin{defin}
    Given $p \in \omega^*$, we say that a topological space $X$ is $p-$compact if every sequence in $X$ has a $p-$limit.
    \end{defin}
    
    Every $p-$compact space is countably compact. We already know that Tychonoff's theorem stands for $p-$compactness. Thereby, every power of a $p-$compact space is countably compact. In addition, by Ginsburg and Saks' theorem \footnote{Theorem 2.6 of \cite{G}.}, $X^{2^{\mathfrak c}}$ is countably compact if, and only if, $X$ is $p-$compact for some $p \in \omega^*$.
    
    On other hand, we will need the following class of ultrafilters.
	
	\begin{defin}
    We say that $p \in \omega^*$ is selective if, for every partition $\lbrace P_n \mid n \in \omega \rbrace$ of $\omega$, exists $m \in \omega$ such that $P_m \in p$ or exists $B \in p$ such that $\vert B \cap P_n \vert \leq 1$, for each $n \in \omega$. 
    \end{defin}
    
    \begin{prop} \label{ufs}
    Let $p$ be a selective ultrafilter. If $\lbrace B_n \mid n \in \omega  \rbrace \subset p$, then there exists an increasing sequence $\lbrace a_n \mid n \in \omega \rbrace \in p$ such that $a_n \in B_n$, for each $n \in \omega$. 
    \end{prop}
    
    Now, we present a pre-order on $\omega^*$ known as the Rudin-Keisler order.
    
    \begin{defin}
    Given $p, q \in \omega^*$, we say that $p \leq q$ if there exists a function $f : \omega \to \omega$ such that $\beta f(q) = p$, where $\beta f$ is the \u Cech-Stone extension of $f$.
    \end{defin}

    The next result gives us some combinatorial properties of many incomparable selective ultrafilters.  

    \begin{lem}[Lemma 3.6, \cite{Tomita05}] \label{variosuf}
    Let $\lbrace p_{j} \mid j \in \omega \rbrace$ be a family of incomparable selective ultrafilters. For each $j \in \omega$, let $\left( a^j_k \right)_{k \in \omega} \subset \omega$ be an increasing sequence such that $\lbrace a^j_k \mid k \in \omega  \rbrace \in p_j$. If $k < a^j_k$, for any $j,k \in \omega$, then exists $\lbrace I_j \mid j \in \omega \rbrace$ a disjoint family of subsets of $\omega$ such that 
        \begin{enumerate}
            \item[i)] $\lbrace a^j_k \mid k \in I_j  \rbrace \in p_j$, 
            \item[ii)] $\lbrace [k,a^j_k] \mid k \in I_j, j \in \omega \rbrace$ is a pairwise disjoint family.
        \end{enumerate}
    \end{lem}

\section{Constructing homomorphisms by arc functions}
	
	In \cite{Tomita21} it was defined a $p-$compact group topology on $\mathbb Q ^{(\kappa)}$, for a selective ultrafilter $p$, through a family of homomorphisms from $\mathbb Q ^{(\kappa)}$ to $\mathbb T$ such that assigns $p-$limits for a basis of its ultrapower. The main idea used for it lies in the construction of an appropriate sequence of arc functions by means of arc equations and rational stacks.   
	
	Here, assuming the existence of $\kappa$ incomparable selective ultrafilters, we will construct homomorphisms from $\mathbb Q ^{(\kappa)}$ to $\mathbb T$ such that assigns $p_{\xi}-$limits for basis of many ultraprowers at the same time, for every incomparable selective ultrafilter $p_{\xi}$. To achieve this, we will make use of Lemma \ref{variosuf} as it was done in \cite{Tomita05}.
	
	Thus, we introduce the following concepts as in \cite{Tomita21}.

    \begin{defin}
        Let $\mathbb{B} = \lbrace I + \mathbb{Z} \mid \emptyset \neq I \subset \mathbb{R} \: \textrm{is an open interval} \rbrace$ be the family of all nonempty open arcs in $\mathbb{T}$. We will call $\varphi : \kappa \to \mathbb{B}$ an arc function if its support, denoted by $\supp{\varphi} = \lbrace \alpha \in \kappa \mid \varphi (\alpha) \neq \mathbb{T}  \rbrace$, is finite. In addition, if $\delta \in \left(0, \frac{1}{2} \right)$ and the arcs $\varphi(\alpha)$ have length $\delta$, for every $\alpha \in \supp{\varphi}$, then we will call it an $\delta -$arc function.
        
        Given two arc functions $\psi$ and $\varphi$, we will say that $\psi \leq \varphi$ if, for every $\alpha \in \kappa$, $\psi(\alpha) = \varphi(\alpha)$ or $\overline{\psi(\alpha)} \subseteq \varphi(\alpha)$.
    \end{defin}
    
    \begin{defin} 
    Let $\phi$ be an arc function. If $S \in \mathbb{N}$, then $S.\phi$ is the arc function defined by ($S.\phi)(\mu)=S.\phi(\mu)$ for every $\mu \in \kappa$. For any $a \in \mathbb Z^{(\kappa)}$, we set $\phi(a) = \sum\limits_{\xi \in \supp{a}} a(\xi) \phi(\xi)$.
	\end{defin}

	\begin{defin} 
		An arc equation is a quintuple $(\phi, A, \mathcal A, S, U)$ where $\phi$ is an arc function,  $A\subseteq \omega$, ${\mathcal A} \subseteq ({\mathbb Z}^{(\kappa)})^A$, $S$ is a positive integer and $U=(U_f)_{f \in \mathcal A}$ is a family of elements of ${\mathbb B}$.
		
		Given $n\in A$, we say that an arc function $\psi$ is an $n$-solution for $(\phi, A, \mathcal A, S, U)$ if $S\psi \leq \phi$ and $\sum\limits_{\mu \in \supp f(n)} \psi(f(n)) \subseteq U_f$, for each $f\in {\mathcal A}$. 
	\end{defin}
	
	\begin{defin} A rational stack is a nonuple $\langle   \mathcal B, \nu, \zeta, K, A, k_0, k_1, l, T\rangle$ conformed by
		\begin{itemize}
			\item $T$ is a positive integer,
			\item $A\subseteq \omega$ is infinite,
			\item $k_0\leq k_1$ are natural numbers with $k_1>0$,
			\item $l:k_1\rightarrow \omega$,
			\item $\nu:k_0\rightarrow \kappa$,
			\item $\zeta:k_1\rightarrow \kappa^\omega$,
			\item $K:\omega\rightarrow \omega\setminus 2$ is such that for every $n \in A$, $n!T\mid K_n$,
			\item $\mathcal B= \lbrace {{\mathcal B}}_{i,j} \mid i< k_1, j < l_i \rbrace$, where $\mathcal B_{i, j}\subseteq \left( \mathbb{Z}^{(\kappa)} \right)^{\omega}$ is finite,
		\end{itemize}
	satisfying the following:	
		\begin{enumerate}[label=\roman*)]
			\item $\zeta_i(n)=\nu_i$ for every $i<k_0$ and $n \in A$,
			\item  The elements of $\lbrace \nu_i \mid i<k_0 \rbrace$ and $\lbrace \zeta_j(n) \mid k_0\leq j<k_1, n \in A \rbrace$ are pairwise distinct,
			\item $\zeta_i(n) \in \supp h(n)$, for each $ i<k_1$, $j<l_i$, $h \in {\mathcal B}_{i,j}$ and $n\in A$,
			\item $\zeta_{i}(n) \notin \supp h(n)$, for each, $i<i_*<k_1$, $j< l_{i_*}$ and $h \in {\mathcal B}_{i_*,j}$ and $n \in A$,
			\item $\left( \frac{h(n)(\zeta_i(n))}{K_n}\right)_{n\in A}$ converges monotonically to $+\infty$, $-\infty$ or a real number, for each $i< k_1$, $j<l_i$ and $h\in {\mathcal B}_{i,j}$,
			\item For every $i<k_1, j<l_i$, there exists $h_* \in \mathcal B_{i, j}$ such that for every $h \in \mathcal B_{i, j}$, $\left(\frac{h(n)(\zeta_i(n))}{h_{*}(n)(\zeta_i(n))}\right)_{n\in A}$ converges to a real number $\theta_{h_*}^h$ and $(\theta_{h_*}^h:\,h \in {\mathcal B}_{i,j})$ is linearly independent (as a $\mathbb Q$-vector space),
			\item For each $i<k_1$, $j'<j<l_i$, $h \in {\mathcal B}_{i,j}$ and $h'\in {\mathcal B}_{i,j'}$, $\left(\frac{h(n)(\zeta_i(n))}{h'(n)(\zeta_i(n))}\right)_{n\in A}$ converges monotonically to $0$,
			\item For each $ i<k_0$ there exists $j<l_i$ such that $\frac{K}{T} . \chi_{\vec{\nu_i}}\in{\mathcal B}_{i,j}$,
			\item $\left(|h(n)(\zeta_i(n))|\right)_{n\in A}$ is strictly increasing, for each $i<k_1, j<l_i$ and $h \in {\mathcal B}_{i,j}$,
			\item  For each $i<k_1, j<l_i$ and distinct $h, h_* \in {\mathcal B}_{i,j}$, either 
			    \begin{itemize}
				    \item $|h(n)(\zeta_i(n))|> |h_*(n)(\zeta_i(n))|$  for each  $n \in A$, or
				    \item   $|h(n)(\zeta_i(n))|= |h_*(n)(\zeta_i(n))|$  for each  $n \in A$, or
				    \item  $|h(n)(\zeta_i(n))|< |h_*(n)(\zeta_i(n))|$  for each  $n \in A$,
			    \end{itemize}
			\item For all $\mu \in \kappa$, $i \in \omega$ such that $k_0\leq i<k_1$ and $g \in \bigcup_{j<l_i}B_{i, j}$, if $\{n \in \omega: \mu \in \supp g(n)(\mu)\}$ then  $\left( \frac{g(n)(\mu)}{K_n}\right)_{n\in A}$ is constant.
		\end{enumerate}
	\end{defin}

    The next results will be useful for our purpose.

    \begin{prop}[Lemma 4.1, \cite{Tomita21}] \label{lema3.1}
    Let $p$ be a selective ultrafilter. Given $B \in p$ and $\mathcal F$ a finite subset of $\left( \mathbb{Q}^{(\kappa)} \right)^{\omega}$ such that $\lbrace [f]_{p} \mid f \in \mathcal{F} \rbrace \cup \lbrace [\chi_{ \overrightarrow{\alpha}}]_{p} \mid \alpha \in \kappa \rbrace$ is linearly independent, then there exists a rational stack $S$ and functions $\mathcal{M} : \mathcal{A} \times \mathcal{C}  \to \mathbb{Z}$ and $\mathcal{N} : \mathcal{C} \times \mathcal{A}  \to \mathbb{Z}$, where $\mathcal{A} = \mathcal{F} \cup \lbrace \chi_{\overrightarrow{\nu_i}} \mid i < k_0  \rbrace$ and $\mathcal{C} = \dfrac{\bigcup_{i \in k_1, j \in l_i} B_{i,j} }{K}$, such that 
        \begin{enumerate}
            \item[a)] $A \subseteq B$, $K . \mathcal{A} \subset \left( \mathbb{Z}^{(\kappa)} \right)^{\omega}$, $K . \mathcal{C} \subset \left( \mathbb{Z}^{(\kappa)} \right)^{\omega}$,
            \item[b)] $\lbrace [f]_p \mid f \in \mathcal{A} \rbrace$ and $\lbrace [h]_p \mid h \in \mathcal{C} \rbrace$ generate the same vector subspace, 
            \item[c)] $f_n = \displaystyle\sum_{h \in \mathcal{C}} \mathcal{M}_{f,h} h_n$, for every $n \in A$ and $f \in \mathcal{A}$, 
            \item[d)] $h_n = \dfrac{1}{T^2} \displaystyle\sum_{h \in \mathcal{C}} \mathcal{N}_{h,f} f_n$, for every $n \in A$ and $h \in \mathcal{C}$. 
        \end{enumerate}
    \end{prop}
    
    \begin{prop}[Lemma 4.2, \cite{Tomita21}] \label{lema3.2}
    Let $\mathcal{S}$, $\mathcal{A}$, $\mathcal{C}$, $\mathcal{M}$, $\mathcal{N}$ be as in Proposition \ref{lema3.1}. Given $\varepsilon > 0$ and $D \subset \kappa$ finite, there exists $B \subset A$ cofinite in $A$ and a sequence of positive real numbers $(\gamma_n)_{n \in B}$ such that:
    
    For any $n \in B$, for any family $W = (W_h)_{h \in \mathcal{C}}$ of open arcs of length $\varepsilon$ and for every $\varepsilon -$arc function $\psi$ with $\supp{\psi} \subseteq D \setminus \lbrace \nu_i \mid i < k_0 \rbrace$, exists an $n -$solution of length $\gamma_n$ for the arc equation $(\psi, B, K. \mathcal{C}, K_n, W)$. 
    \end{prop}

    \begin{prop}[Lemma 4.3, \cite{Tomita21}] \label{lema3.3}
    Let $\mathcal{S}$, $\mathcal{A}$, $\mathcal{C}$, $\mathcal{M}$, $\mathcal{N}$ be as in Proposition \ref{lema3.1}. Let $\delta$ be a positive number such that $\varepsilon = \dfrac{\delta}{\displaystyle\sum_{f \in \mathcal{A}, h \in \mathcal{C}} \vert \mathcal{M}_{f,h} \vert } < \frac{1}{2}$. If $U = (U_f)_{f \in \mathcal{A}}$ is a family of open arcs of length $\delta$ and $\Gamma$ be an $\delta -$arc function, with $\lbrace \nu_i \mid i < k_0 \rbrace \subset \supp{\psi}$, such that $\Gamma(\nu_i) = U_{\chi_{ \overrightarrow{\nu_i}}}$, for every $i < k_0$, then exists a family $W = (W_h)_{h \in \mathcal{C}}$ of open arcs of length $\varepsilon$ and an $\varepsilon -$arc function $\psi$, with $\supp{\psi} = \supp{\Gamma} \setminus \lbrace \nu_i \mid i < k_0 \rbrace$, such that, for each $n \in A$, every $n-$solution for $(\psi, A, K. \mathcal{C}, K_n, W)$ is an $n-$solution for $(\Gamma, A, K. \mathcal{A}, K_n, U)$.  
    \end{prop}
    
    Now we will construct a suitable sequence of arc equations that allows us to define partial homomorphisms.

    \begin{lem} \label{stacksequences}
    Let  $\lbrace p_{n} \mid n \in \omega \rbrace$ be a family of incomparable selective ultrafilters and  $(\mathcal{F}^n)_{n \in \omega}$ be a countable family of $\left( \mathbb{Q}^{(\kappa)} \right)^{\omega}$ such that, for each $n \in \omega$, $\lbrace [f]_{p_{n}} \mid f \in \mathcal{F}^n \rbrace \cup \lbrace [ \chi_{\overrightarrow{\mu}} ]_{p_{n}} \mid \mu \in \kappa \rbrace$ is linearly independent in $Ult_{p_{n}} \left( \mathbb{Q}^{(\kappa)} \right)$. 
    
    Let $d \in \mathbb{Q}^{(\kappa)}$ be non-zero and $C \subset \kappa$ be a countable subset such that $\supp d \cup \omega \cup \displaystyle\bigcup_{n, k \in \omega} \lbrace \supp{f_k} \mid f \in \mathcal{F}^n \rbrace \subset C$. In addition, for any $f \in \displaystyle\bigcup_{n \in \omega} \mathcal{F}^n$, fix some $\xi_f \in C$. 
    
    For every $n \in \omega$, let $(\mathcal{F}^{k,n})_{k \in \omega}$ be a increasing sequence of finite subsets whose union is $\mathcal{F}^n$. Let $(C_n)_{n \in \omega}$ be a increasing sequence of finite subsets whose union is $C$ and $\displaystyle\bigcup_{i,k \leq n+1} \lbrace \supp{f_k} \mid f \in \mathcal{F}^{n+1,i} \rbrace \subset C_n$. Then, there exists
    \begin{itemize}
        \item an increasing sequence of positive integers $(i_n)_{n \in \omega}$,
        \item a function $r : \omega \to \omega$,
        \item an increasing sequence of positive integers $(b_n)_{n \in \omega}$ such that $\lbrace b_k \mid k \in r^{-1}(n) \rbrace \in p_n$, for any $n \in \omega$, 
        \item a sequence of positive integers $(Q_n)_{n \in \omega \cup \lbrace -1 \rbrace}$, with $Q_{-1} = 1$, such that $n! \mid Q_n$ and $Q_{n} \mid Q_{n+1}$, for any $n \in \omega$, 
        \item a sequence of arc functions $(\varrho_{i_n})_{n \in \omega}$ such that $0 \notin \overline{\varrho_{i_0}(d)}$ and, for any $n \in \omega$, $\supp{\varrho_{i_{n}}} \cup C_{i_{n+1}}  \subseteq \supp{\varrho_{i_{n+1}}}$,
    \end{itemize}
    satisfying, for any $0 \leq m' \leq m$, the following   
    \begin{enumerate}
        \item[i)] for every $\xi \in \supp{\varrho_{i_{m+1}}}$, $\frac{Q_{m}}{Q_{m'-1}} \varrho_{i_{m+1}}(\xi) $ has length less or equal than $\frac{1}{2^{i_m}}$ and its contained in $\varrho_{i_{m'}}(\xi)$,
        \item[ii)] for every $f \in \mathcal{F}^{i_m,r_m}$, we have $\frac{Q_{m}}{Q_{m-1}} \varrho_{i_{m+1}}(f_{b_m}) \subseteq \varrho_{i_{m}}(\xi_f)$, 
        \item[iii)] for every $f \in \mathcal{F}^{i_m,r_m}$, we have $\frac{Q_{m}}{Q_{m'-1}} \varrho_{i_{m+1}}(f_{b_m}) \subseteq \varrho_{i_{m'}}(\xi_f)$.
    \end{enumerate}
    \end{lem}

    \begin{proof}
    We begin establishing the next property.
    \begin{claim}
    Given $m \in \mathbb{N}$ and $k < m$, there exists
        \begin{itemize}
            \item rational stacks $\mathcal{S}^{m,k} = \langle \mathcal{B}^{m,k}, \nu^{m,k}, \zeta^{m,k}, K^{m,k}, A^{m,k}, k^{m,k}_0, k^{m,k}_1, l^{m,k}, T^{m,k}  \rangle$,  
            \item a positive number $\delta_{m-1}$ less or equal than $\frac{1}{2^{m}}$,
            \item $B^{m,k} \in p_k$, with $B^{m,k} \subset A^{m,k} \setminus (m-k)$,
            \item a sequence of positive real numbers $(\gamma^{m,k}_n)_{n \in B^{m,k}}$,
        \end{itemize}
    such that
        \begin{enumerate}
            \item[(*)] If $U^{m,k} = (U_f)_{f \in \: \mathcal{A}^{m,k}}$ is a family of $\delta_{m-1} -$arc functions, where $\mathcal{A}^{m,k} = \mathcal{F}^{m,k} \cup \left\lbrace \chi_{\overrightarrow{\nu^{m,k}_i}} \mid i \in k^{m,k}_0 \right\rbrace$, and $\psi$ is an $\delta_{m-1} -$arc function such that $U_{\chi_{\overrightarrow{\nu^{m,k}_i}}} = \psi \left( \nu^{m,k}_i \right)$, for each $i \in \: k^{m,k}_0$, then, for every $n \in B^{m,k}$, exists an $n-$solution of length $\gamma^{m,k}_n$ for $\left( \psi, B^{m,k}, K^{m,k}.\mathcal{A}^{m,k}, K^{m,k}_n, U^{m,k} \right)$.
        \end{enumerate}
    \end{claim}
\noindent
    \textbf{Proof of Claim.} By Proposition \ref{lema3.1} for $\omega$ and $\mathcal{F}^{1,0}$, there exists $\mathcal{S}^{1,0}$, $\mathcal{A}^{1,0}$, $\mathcal{C}^{1,0}$, $\mathcal{M}^{1,0}$, $\mathcal{N}^{1,0}$. Fix $\delta_0 \in \left( 0, \frac{1}{2} \right)$. Note that we can apply Proposition \ref{lema3.3} to $\varepsilon_0 = \dfrac{\delta_0}{\displaystyle\sum_{f \in \mathcal{A}^{1,0}, h \in \mathcal{C}^{1,0}} \vert \mathcal{M}^{1,0}_{f,h} \vert}$. Set $D_0 = C_0 \cup \left\lbrace \nu^{1,0}_i \mid i \in k^{1,0}_0 \right\rbrace$, so exists $B^{1,0} \subset A^{1,0} \setminus 1$ cofinite in $A^{1,0} \setminus 1$ and $\left( \gamma^{1,0}_n \right)_{n \in B^{1,0}}$ be as in Proposition \ref{lema3.2}. Clearly $(*)$ is true. 
    
    Given $m$, assume that exists $ \lbrace \mathcal{S}^{m,k} \mid k < m \rbrace$, $\delta_{m-1} > 0$, $\lbrace B^{m,k} \mid k < m \rbrace$, $(\gamma^{m,k}_n)_{n \in B^{m,k}}$ such that $(*)$ is satisfied. 
    
    Fix $k < m+1$. If $k < m$, for $B^{m,k} \in p_k$ and $\mathcal{F}^{m+1,k}$, there exists $\mathcal{S}^{m+1,k}$, $\mathcal{A}^{m+1,k}$, $\mathcal{C}^{m+1,k}$, $\mathcal{M}^{m+1,k}$, $\mathcal{N}^{m+1,k}$ be as in Proposition \ref{lema3.1}. If $k = m$, consider $\omega \in p_m$ and $\mathcal{F}^{m+1,m}$, so analogously exists $\mathcal{S}^{m+1,m}$, $\mathcal{A}^{m+1,m}$, $\mathcal{C}^{m+1,m}$, $\mathcal{M}^{m+1,m}$, $\mathcal{N}^{m+1,m}$. 

    Let $\delta_{m} = \frac{2^{-(m+1)}}{ \displaystyle\prod\limits_{i<j \leq m+1, t\leq m+2} K^{i,j}_{t} }. \min \left\lbrace \delta_{m-1}, \gamma^{i,j}_n \mid i < j \leq m+1, n \in B^{i,j} \cap (m+2)  \right\rbrace$. Note that we can apply Proposition \ref{lema3.3} to $\varepsilon^{m+1,k} = \dfrac{\delta_{m}}{\displaystyle\sum_{f \in \mathcal{A}^{m+1,k}, h \in \mathcal{C}^{m+1,k}} \vert \mathcal{M}^{m+1,k}_{f,h} \vert}$. Set $D_{m+1} = D_m \cup C_{m+1} \cup \displaystyle\bigcup_{l \leq m} \left\lbrace \nu^{m+1,l}_i \mid i \in k^{m+1,l}_0  \right\rbrace$. So, for each $k < m$, exists $B^{m+1,k} \subset A^{m+1,k} \setminus (m+1-k)$ cofinite in $A^{m+1,k} \setminus (m+1-k)$ and $(\gamma^{m+1,k}_n)_{n \in B^{m+1,k}}$ as in Proposition \ref{lema3.2}. Analogously, for $\omega$ and $\mathcal{F}^{m+1,m}$, exists $B^{m+1,m} \subset A^{m+1,m} \setminus 1$ cofinite in $A^{m+1,m} \setminus 1$ and $(\gamma^{m+1,m}_n)_{n \in B^{m+1,m}}$. 
    
    Finally, fix $k < m+1$. Let $U^{m+1,k} = (U_f)_{f \in \mathcal{A}^{m+1,k}}$ be a family of $\delta_m -$arc functions and $\psi$ be a $\delta_m-$arc function such that $U_{\chi_{\overrightarrow{\nu^{m+1,k}_i}}} = \psi \left(\nu^{m+1,k}_i \right)$, for each $i \in \: k^{m+1,k}_0$. Then, by Proposition \ref{lema3.3}, we know that, for every $n \in B^{m+1,k}$, exists an $n-$solution of length $\gamma^{m+1,k}_n$ for $\left( \psi, B^{m+1,k}, K^{m+1,k}.\mathcal{A}^{m+1,k}, K^{m+1,k}_n, U^{m+1,k} \right)$. \\
    \qed
    
    For each $k \in \omega$, $\left( B^{n+k+1,k} \right)_{n \in \omega}$ is a decreasing sequence of $p_k$. So, by Proposition \ref{ufs}, exists $(a^k_n)_{n \in \omega}$ such that $a^k_n \in B^{n+k+1,k}$, with $a^k_n > n$, and $\lbrace a^k_n \mid n \in \omega  \rbrace \in p_k$. In virtue of Lemma \ref{variosuf}, exists a disjoint family $\lbrace I_j \mid j \in \omega \rbrace$ of subsets of $\omega$ such that 
        \begin{itemize}
            \item $\lbrace a^k_n \mid n \in I_k  \rbrace \in p_k$,
            \item $\lbrace [n,a^k_n] \mid n \in I_k, k \in \omega \rbrace$ is a disjoint family.
        \end{itemize}
    Let $\lbrace i_t \mid t \in \omega \rbrace$ be an increasing enumeration of $\displaystyle\bigcup_{j \in \omega} I_j$. Define $r : \omega \to \omega$ by $r_n = k$ if, and only if, $i_n \in I_k$. Now, set $b_t = a^{r_t}_{i_t}$, for every $t \in \omega$. Note that $(b_t)_{t \in \omega}$ is increasing and $\lbrace b_k \mid k \in r^{-1}(j) \rbrace \in p_j$, for every $j \in \omega$. 
    
    Inductively we can construct $\phi_{i_m}$, $\varrho_{i_m}$ an $\delta_{i_m-1} -$arc function and $U^{i_m, r_m}$ such that, for each $0 \leq m' \leq m$, satisfying
        \begin{enumerate}
            \item[a)] $\supp{\varrho_{i_m}} \subset \supp{\varrho_{i_{m+1}}}$ and $C_{i_{m}} \subset \supp{\varrho_{i_{m}}}$,
            \item[b)] $U^{i_m,r_m} = (U_f)_{f \in \: \mathcal{A}^{i_m,r_m}}$, where $U_f = \varrho_{i_m}(\xi_f)$, if $f \in \mathcal{F}^{i_m,r_m}$, and $U_f = \varrho_{i_m}(\nu^{i_m,r_m}_i)$ if $f = \chi_{\overrightarrow{\nu^{i_m,r_m}_i}}$, for every $i < k^{i_m,r_m}_0$,
            \item[c)] $\phi_{i_m}$ is an $b_{m} -$solution of length $\gamma^{i_m,r_m}_{b_m}$ for the arc equation $$\left( \varrho_{i_{m}}, B^{i_m,r_m}, K^{i_m,r_m}.\mathcal{A}^{i_m,r_m}, K^{i_m,r_m}_{b_m}, U^{i_m,r_m} \right),$$
            \item[d)] $\varrho_{i_{m+1}} \leq \phi_{i_m}$,
            \item[e)] for each $\xi \in \supp{\varrho_{i_{m+1}}}$, $\displaystyle\prod\limits_{k=m'}^{m} K^{i_{k}, r_k}_{b_{k}} \varrho_{i_{m+1}}(\xi)$ has length less or equal than $\frac{1}{2^{i_m}}$ and its contained in $\varrho_{i_{m'}}(\xi)$,
            \item[f)] for each $f \in \mathcal{F}^{i_m,r_m}$, we have $K^{i_{m}, r_m}_{b_{m}} \varrho_{i_{m+1}}(f_{b_m}) \subseteq \varrho_{i_{m}}(\xi_f)$, 
            \item[g)] for each $f \in \mathcal{F}^{i_m,r_m}$, we have $\displaystyle\prod\limits_{k=m'}^{m} K^{i_{k}, r_k}_{b_{k}} \varrho_{i_{m+1}}(f_{b_m}) \subseteq \varrho_{i_{m'}}(\xi_f)$.
        \end{enumerate}
    In fact, let $\varrho_{i_0}$ be an $\delta_{i_o} -$arc function such that $C_{i_o} \cup \lbrace \nu^{i_0,r_0}_{i} \mid i < k^{i_0,r_0}_{0} \rbrace \subseteq \supp{\varrho_{i_o}}$ and $0 \notin \overline{\varrho_{i_0}(d)}$. Define $U^{i_0,r_0} = (U_f)_{f \in \mathcal{A}^{i_{0},r_{0}}}$ by $U_f = \varrho_{i_0}(\xi_f)$, for each $f \in \mathcal{F}^{i_{0},r_{0}}$, and $U_{\chi_{\overrightarrow{\nu^{i_0,r_0}_{i} }}} = \varrho_{i_0}(\nu^{i_0,r_0}_{i})$, for each $i \in k^{i_0,r_0}_{0}$. By $(*)$, $\left( \varrho_{i_0}, B^{i_{0},r_{0}}, K^{i_0,r_0}.\mathcal{A}^{i_0, r_0}, K^{i_{0},r_{0}}_{b_{0}}, U^{i_0,r_0} \right)$ have an $b_{0}-$solution $\phi_{i_0}$ of length $\gamma^{i_{0}, r_0}_{b_{0}}$. Thereby, $K^{i_{0},r_{0}}_{b_{0}}. \phi_{i_0}(f_{b_0}) \subseteq \varrho_{i_0}(\xi_f)$, for every $f \in \mathcal{F}^{i_{0},r_{0}}$.  
    
    Assume that we have $\phi_{i_m}$, $\varrho_{i_m}$ and $U^{i_m, r_m}$. Let $\varrho_{i_{m+1}}$ be an $\delta_{i_{m}} -$arc function such that $\supp{\varrho_{i_{m}}} \cup C_{i_{m+1}} \cup \lbrace \nu^{i_{m+1},r_{m+1}}_{i} \mid i < k^{i_{m+1},r_{m+1}}_{0} \rbrace \subseteq \supp{\varrho_{i_{m+1}}}$ and $\varrho_{i_{m+1}} \leq \phi_{i_m}$. So, we have $(a)$ and $(d)$.
    
    Consequently, note that $K^{i_{m},r_{m}}_{b_{m}}. \varrho_{i_{m+1}}(\xi)  \subseteq \varrho_{i_m}(\xi)$, for every $\xi \in \supp{\varrho_{i_{m+1}}}$, and $K^{i_{m},r_{m}}_{b_{m}}. \varrho_{i_{m+1}}(f_{b_{m}}) \subseteq \varrho_{i_{m}}(\xi_f)$, for every $f \in \mathcal{F}^{i_{m},r_{m}}$. We already have $(f)$. 
    
    In order to obtain $(e)$, the second statement follows from $(c)$ and $(d)$ and then use $(e)$ iteratively. As for the first, follows from the definition of $\delta_{i_m}$. Item $(g)$ follows from $(e)$ and $(f)$.
    
    Now, we define $U^{i_{m+1},r_{m+1}} = (U_f)_{f \in \mathcal{A}^{i_{m+1},r_{m+1}}}$ by $U_f = \varrho_{i_{m+1}}(\xi_f)$, for every $f \in \mathcal{F}^{i_{m+1},r_{m+1}}$, and $U_{\chi_{\overrightarrow{\nu^{i_{m+1},r_{m+1}}_{i} }}} = \varrho_{i_{m+1}}(\nu^{i_{m+1},r_{m+1}}_{i})$, for every $i \in k^{i_{m+1},r_{m+1}}_{0}$. 
    
    By $(*)$, $\left( \varrho_{i_{m+1}}, B^{i_{m+1},r_{m+1}}, K^{i_{m+1},r_{m+1}}.\mathcal{A}^{i_{m+1}, r_{m+1}}, K^{i_{m+1},r_{m+1}}_{b_{m+1}}, U^{i_{m+1},r_{m+1}} \right)$ have an $b_{m+1}-$solution $\phi_{i_{m+1}}$ of length $\gamma^{i_{m+1}, r_{m+1}}_{b_{m+1}}$. With this, we obtain $(b)$ and $(c)$, and therefore the recursion is complete. 
    
    Finally, we put $Q_m = \displaystyle\prod\limits_{k=0}^{m} K^{i_{k}, r_k}_{b_{k}}$, for each $n \in \omega$, and $Q_{-1} = 1$. Hence, items $(i)$, $(ii)$ and $(iii)$ are clearly satisfied.
    \end{proof}

    \begin{lem} \label{homoenume}
    Let  $\lbrace p_{n} \mid n \in \omega \rbrace$ be a family of incomparable selective ultrafilters and  $(\mathcal{F}^n)_{n \in \omega}$ be a countable family of $\left( \mathbb{Q}^{(\kappa)} \right)^{\omega}$ such that, for each $n \in \omega$, $\lbrace [f]_{p_{n}} \mid f \in \mathcal{F}^n \rbrace \cup \lbrace [ \chi_{\overrightarrow{\mu}} ]_{p_{n}} \mid \mu \in \kappa \rbrace$ is linearly independent in $Ult_{p_n} \left( \mathbb{Q}^{(\kappa)} \right)$. 
    
    Let $d \in \mathbb{Q}^{(\kappa)}$ be non-zero and $C \subset \kappa$ be countable subset such that $\supp d \cup \omega \cup \displaystyle\bigcup_{n, k \in \omega} \lbrace \supp{f_k} \mid f \in \mathcal{F}^n \rbrace \subset C$. In addition, for any $f \in \displaystyle\bigcup_{n \in \omega} \mathcal{F}^n$, fix some $\xi_f \in C$. 
    
    Then, there exists a homomorphism $\phi : \mathbb{Q}^{(C)} \rightarrow \mathbb{T}$ such that
        \begin{enumerate}
            \item[a)] $\phi(d) \neq 0$,
            \item[b)] $p_{n}- \lim\limits_{k \to \infty} \phi \left( \frac{1}{N} f_k \right) = \phi \left( \frac{1}{N} \chi_{\xi_f} \right) $, for any $N \in \mathbb{N}$, $f \in \mathcal{F}^n$ and $n \in \omega$.
        \end{enumerate}
    \end{lem}

    \begin{proof}
    Consider $(i_n)_{n \in \omega}$, $r : \omega \to \omega$, $(b_n)_{n \in \omega}$, $(Q_{n})_{n \geq -1}$ and $(\varrho_{i_n})_{n \in \omega}$ as in Lemma \ref{stacksequences}. Given $\alpha \in C$ and $N \in \mathbb{N}$, we will define $ \phi \left( \frac{1}{N} \chi_{\alpha} \right)$ in the following way. 
    
    For every $m \in \omega$ with $\alpha \in \supp{\varrho_{i_m}}$, we set $\phi \left( \frac{1}{Q_m} \chi_{\alpha} \right)$ as the unique element of $\bigcap\limits_{k \geq m} \frac{Q_k}{Q_m}. \varrho_{i_{k+1}}(\alpha)$. Furthermore, for any $m < m'$, we have $ \phi \left( \frac{1}{Q_m} \chi_{\alpha} \right) = \frac{Q_{m'}}{Q_m} \phi \left( \frac{1}{Q_{m'}} \chi_{\alpha} \right)$. 
    
    Now, take some $n_o \in \omega$ such that $\alpha \in \supp{\varrho_{i_{n_0}}}$ and $N \mid Q_{n_0}$, and we define $\phi \left( \frac{1}{N} \chi_{\alpha} \right) = \frac{Q_{n_0}}{N} \phi \left( \frac{1}{Q_{n_0}} \chi_{\alpha} \right)$. Note that $\phi \left( \frac{1}{N} \chi_{\alpha} \right)$ is well-defined because does not depend of $n_0$ and $\phi$ can be extended to a homomorphism. 
    
    For any $n \in \omega$, $f \in \mathcal{F}^n$ and $N \in \mathbb{N}$, we claim that $\left( \phi \left( \frac{1}{N} f_{b_m} \right) \right)_{m \in r^{-1}(n)}$ converges to $\phi \left( \frac{1}{N} \chi_{\xi_f} \right)$. In fact, we fix some $m_0 \in \omega$ such that $f \in \mathcal{F}^{i_{m_0},n}$, $\xi_f \in \supp{\varrho_{i_{m_0}}}$ and $N \mid Q_{m_0-1}$. Then, for every $m \in r^{-1}(n)$ with $m \geq m_0$, we have
        \begin{eqnarray*}
            \phi \left( \frac{1}{N} f_{b_m} \right)  & = & \phi \left( \frac{1}{N} \displaystyle\sum\limits_{\mu \in \supp{f_{b_m}}} f_{b_m}(\mu). \chi_{\mu} \right), \\
            & = & \phi \left( \displaystyle\sum\limits_{\mu \in \supp{f_{b_m}}} \frac{f_{b_m}(\mu)}{N}. \chi_{\mu} \right), \\
            & = & \displaystyle\sum\limits_{\mu \in \supp{f_{b_m}}} \phi \left( \frac{f_{b_m}(\mu)}{N}. \chi_{\mu} \right), \\
            & = & \displaystyle\sum\limits_{\mu \in \supp{f_{b_m}}} \frac{f_{b_m}(\mu)}{N}. Q_m. \phi \left( \frac{1}{Q_m}. \chi_{\mu} \right), \\
            & = & \frac{Q_{m-1}}{N}. K^{i_m,r_m}_{b_m} \displaystyle\sum\limits_{\mu \in \supp{f_{b_m}}} f_{b_m}(\mu) \phi \left( \frac{1}{Q_m}. \chi_{\mu} \right), \\ 
            & \in & \frac{Q_{m-1}}{N}. K^{i_m,r_m}_{b_m} \varrho_{i_{m+1}}(f_{b_m}) \subseteq \frac{1}{N}.Q_{m-1}\varrho_{i_{m}}(\xi_f).
        \end{eqnarray*}
    Clearly $\frac{1}{N}.Q_{m-1}\varrho_{i_{m}}(\xi_f)$ is a neighbourhood of $\phi \left( \frac{1}{N} \chi_{\xi_f} \right)$ and has length less or equal than $\frac{1}{2^{i_{m}}}$. Since $\lbrace b_{k} \mid k \in r^{-1}(n) \rbrace \in p_n$, we conclude that $p_{n}- \lim\limits_{m \to \infty} \phi \left( \frac{1}{N} f_m \right) = \phi \left( \frac{1}{N} \chi_{\xi_f} \right) $. 
    \end{proof}

    Ultimately, we extend the homomorphism above to $\mathbb{Q}^{(\kappa)}$ with respect to $\kappa$ incomparable selective ultrafilters.

    \begin{teo} \label{homo}
    Let $\lbrace p_{\xi} \mid \omega \leq \xi < \kappa \rbrace$ be a family of incomparable selective ultrafilters and $\lbrace f_{\xi} \mid \omega \leq \xi < \kappa \rbrace$ a family of injective sequences of $\mathbb{Q}^{(\kappa)}$ such that $\displaystyle\bigcup_{n \in \omega} \supp{f_{\xi}(n)} \subset  \xi$. In addition, consider a family $\lbrace A_{\xi} \mid \omega \leq \xi < \kappa \rbrace$ of disjoint subsets of $\kappa \setminus \omega$ such that $\lbrace [f_{\beta}]_{p_{\xi}} \mid \beta \in A_{\xi} \rbrace \cup \lbrace [ \chi_{\overrightarrow{\mu}} ]_{p_{\xi}} \mid \mu \in \kappa \rbrace$ is linearly independent in $Ult_{p_{\xi}} \left( \mathbb{Q}^{(\kappa)} \right)$, for every $\xi \in [\omega, \kappa)$. Then, given $d \in \mathbb{Q}^{(\kappa)}$ non-zero, there exists a homomorphism $\varphi_d : \mathbb{Q}^{(\kappa)} \to \mathbb{T}$ such that 
        \begin{enumerate}
            \item[a)] $\varphi_d(d) \neq 0$,
            \item[b)] $p_{\xi}- \lim\limits_{n \to \infty} \varphi_d \left( \frac{1}{N} f_{\beta}(n) \right) = \varphi_d \left( \frac{1}{N} \chi_{\beta} \right) $, for each $N \in \mathbb{N}$, $\beta \in A_{\xi}$ and $\xi \in [\omega, \kappa)$.
        \end{enumerate}
    Furthermore, let $\alpha < \kappa$ be an infinite cardinal and suppose $\lbrace A_{\xi} \mid \omega \leq \xi < \kappa \rbrace \subset \kappa \setminus \alpha$. For every $A \subset \alpha$, exists a homomorphism $\psi_A : \mathbb{Q}^{(\kappa)} \to \mathbb{T}$ such that 
        \begin{enumerate}
            \item[c)] $\left\lbrace \beta \in \alpha \mid \psi_A \left(  \chi_{\beta} \right) = [ \pi ] \right\rbrace = A$ and $\psi_A \mid_{\mathbb{Q}^{(\alpha \setminus A)}} = 0$,
            \item[d)] $p_{\xi}- \lim\limits_{n \to \infty} \psi_A \left( \frac{1}{N}  f_{\beta}(n) \right) = \psi_A \left( \frac{1}{N} \chi_{\beta} \right)$, for each $N \in \mathbb{N}$, $\beta \in A_{\xi}$ and $\xi \in [\omega, \kappa)$.
        \end{enumerate} 
    \end{teo}

    \begin{proof}
    Let $d \in \mathbb{Q}^{(\kappa)}$ non-zero and take a countably subset $C$ of $\kappa$ such that $\omega \cup \supp d \cup \displaystyle\bigcup \lbrace \supp{f_{\theta}(n)} \mid \theta \in C, n \in \omega \rbrace \subset C$. Fix an increasing enumeration $\lbrace \lambda_{\theta} \mid \theta < \kappa \rbrace$ of $\kappa \setminus C$.

    Given $\mu < \kappa$, set $C_{\mu} = C \cup \lbrace \lambda_{\theta} \mid \theta < \mu  \rbrace$ and $I_{\mu} = \lbrace \xi \in \kappa \setminus \omega \mid A_{\xi} \cap C_{\mu} \neq \emptyset \rbrace$. Notice that $C_0 = C$, $I_0$ is countable, $C_{\kappa} =  \kappa$ and $I_{\kappa} = [\omega, \kappa)$. Furthermore, $\displaystyle\bigcup_{n \in \omega} \supp{f_{\theta}(n)} \subset C_{\mu}$, for each $\theta \in A_{\xi} \cap C_{\mu}$ and $\xi \in I_{\mu}$.

    For every $\mu < \kappa$, we will construct inductively a homomorphism $\phi_{\mu} : \mathbb{Q}^{(C_{\mu})} \rightarrow \mathbb{T}$ such that
    \begin{enumerate}
        \item[i)] $\phi_{\mu}(d) \neq 0$,
        \item[ii)] $\phi_{\beta} \subseteq \phi_{\mu}$, if $\beta \leq \mu$,
        \item[iii)] $p_{\xi}- \lim\limits_{n \to \infty} \phi_{\mu} \left( \frac{1}{N} f_{\delta}(n) \right) = \phi_{\mu} \left( \frac{1}{N} \chi_{\delta} \right)$, for each $N \in \mathbb{N}$, $\delta \in A_{\xi} \cap C_{\mu}$ and $\xi \in I_{\mu}$.
    \end{enumerate}  

    In fact, $\phi_0$ is obtained by Lemma \ref{homoenume}. Given $\mu < \kappa$, suppose we have $\phi_{\theta}$, for every $\theta < \mu$. In case $\mu$ is not an ordinal limit, exists $\nu < \kappa$ such that $\mu = \nu +1$. We have $\mathbb{Q}^{(C_{\mu})} = \mathbb{Q}^{(C_{\nu})} \oplus \langle \lbrace \chi_{\lambda_{\nu}} \rbrace \rangle$. If $\lambda_{\nu} \notin \displaystyle\bigcup_{\xi \in \kappa \setminus \omega} A_{\xi}$, we define $h = 0$. Other case, exists $\xi \in \kappa \setminus \omega$ such that $\lambda_{\nu} \in A_{\xi}$. We define $h \left( \frac{1}{N} \chi_{\lambda_{\nu}} \right) = p_{\xi} - \lim\limits_{n \to \infty} \phi_{\nu} \left( \frac{1}{N} f_{\lambda_{\nu}}(n) \right)$, for every $N \in \mathbb{N}$. In that way, set $\phi_{\nu+1} = \phi_{\nu} \oplus h$, which satisfies the desire conditions.  

    In case $\mu$ be an ordinal limit, notice that $C_{\mu} = \displaystyle\bigcup_{\theta < \mu} C_{\theta}$ and, as a consequence, $I_{\mu} = \displaystyle\bigcup_{\theta < \mu} I_{\theta}$ and $\mathbb{Q}^{(C_{\mu})} = \displaystyle\bigcup_{\theta < \mu} \mathbb{Q}^{(C_{\theta})}$. We define $\phi_{\mu} =  \displaystyle\bigcup_{\theta < \mu} \phi_{\theta}$, which satisfies the desire conditions. Thus, $\varphi_d = \phi_{\kappa}$ is the homomorphism that satisfies $(a)$ and $(b)$. 

    Finally, let $\alpha < \kappa$ be an infinite cardinal and suppose $\lbrace A_{\xi} \mid \omega \leq \xi < \kappa \rbrace \subset \kappa \setminus \alpha$. Given $A \subset \alpha$, define $\psi_A : \mathbb{Q}^{(\alpha)} \to \mathbb{T}$ such that $\psi_A \left( \frac{1}{N} \chi_{\beta} \right) = \big[ \frac{1}{N} \pi \big]$, if $\beta \in A$ and $N \in \mathbb{N}$, and $\psi_A \left( \frac{1}{N} \chi_{\beta} \right) = 0$, if $\beta \in \alpha \setminus A$ and $N \in \mathbb{N}$. Analogously, we can extend $\psi_A$ to $\mathbb{Q}^{(\kappa)}$ such that conditions $(c)$ and $(d)$ are satisfied.
    \end{proof}

\section{Torsion-free topological groups with respect to Comfort's question}

    Our goal in this section is to show that the existence of $2^{\mathfrak{c}}$ incomparable selective ultrafilters and a cardinal arithmetic implies that, for every infinite cardinal $\alpha \leq 2^{\mathfrak c}$, there exists a torsion-free Abelian topological group as in the Comfort's question. To get this, we will adapt the ideas developed in \cite{Tomita05} making use of the homomorphisms obtained by Theorem \ref{homo}.

    For now, we will establish the next result that connects ultrapowers on $p$ to $p-$compactness.

    \begin{prop} \label{plim}
    Let $p$ be a selective ultrafilter, $G$ a topological group, and $\varphi : \mathbb{Q}^{(\kappa)} \to G$ a group homomorphism. Let $M$ be a linear subspace of $\left( \mathbb{Q}^{(\kappa)} \right)^{\omega}$ and $\lbrace [f]_p \mid f \in B \rbrace$ a basis of $Ult_p(M)$. If $\left( \varphi \left( \frac{1}{N} f(n) \right) \right)_{n \in \omega}$ have $p-$limit in $G$, denoted by $a_{N,f}$, for each $f \in B$ and $N \in \mathbb{N}$, then $p - \lim\limits_{n \to \infty} \varphi(h(n)) \in \langle \lbrace a_{N, f} \mid N \in \mathbb{N}, f \in B \rbrace \rangle$, for every $h \in M$.
    \end{prop}

    \begin{proof}
    Given $h \in M$, there exists $\left\lbrace q_i \mid i < k \right\rbrace \subset \mathbb{Q}$ and $\lbrace f_i \mid i < k \rbrace \subset B$ such that $[h]_p = \displaystyle\sum_{i < k} q_i [f_i]_p$. So, there exists $C \in p$ such that $h(n) = \displaystyle\sum_{i < k} q_i f_i(n)$, for every $n \in C$. Furthermore, consider $q_i = \frac{m_i}{n_i}$, where $m_i \in \mathbb{Z}$ and $n_i \in \mathbb{N}$. As a consequence, $\varphi (h(n)) = \displaystyle\sum_{i < k} m_i \varphi \left( \frac{1}{n_i} f_i(n) \right)$, for every $n \in C$. Thus, $p - \lim\limits_{n \to \infty} \varphi(h(n)) = \displaystyle\sum_{i < k} m_i a_{n_i, f_i} \in \langle \lbrace a_{N, f} \mid N \in \mathbb{N}, f \in B \rbrace \rangle$.
    \end{proof}

    In the remaining of this section, we assume the existence of $2^{\mathfrak{c}}$ incomparable selective ultrafilters and $2^{< 2^{\mathfrak{c}}} = 2^{\mathfrak{c}}$. We will  consider two cases.  
 
    \subsection{Case $\alpha = 2^{\mathfrak{c}}$}

    \begin{prop} \label{Prop2c}
    Let $\lbrace p_{\xi} \mid \omega \leq \xi < 2^{\mathfrak{c}} \rbrace$ be a family of incomparable selective ultrafilters. Given a family $\lbrace g_{\xi, \gamma} \mid \omega \leq \xi < 2^{\mathfrak{c}}, \: \omega \leq \gamma < 2^{\mathfrak{c}} \rbrace$ of injective sequences of $\displaystyle\bigcup_{\theta \in [\omega, 2^{\mathfrak{c}})} \left( \mathbb{Q}^{(2^{\mathfrak{c}})} \right)^{\theta}$ such that $I_{\xi} = \displaystyle\bigcup \lbrace \supp{g_{\xi, \beta}(n)} \mid \beta < \gamma, \: n \in \omega \rbrace \subset \xi$, \footnote{Here $g_{\xi, \gamma} = ( g_{\xi, \beta} )_{\beta < \gamma}$ is an injective sequence in $\left( \mathbb{Q}^{(2^{\mathfrak{c}})} \right)^{\gamma}$.} there exists an increasing sequence $\lbrace \delta_{\xi} \mid \omega \leq \xi < 2^{\mathfrak{c}} \rbrace$ in $2^{\mathfrak{c}} \setminus \omega$, a family $\lbrace A_{\xi} \mid \omega \leq \xi < 2^{\mathfrak{c}} \rbrace$ of disjoint subsets of $2^{\mathfrak{c}} \setminus \omega$, with $\vert A_{\xi} \vert < 2^{\mathfrak{c}}$, and a linearly independent subset $\lbrace x_{m, \beta} \mid \: m \in \mathbb{N}, \beta \in 2^{\mathfrak{c}} \rbrace$ of $\mathbb{T}^{2^{\mathfrak{c}}}$ such that
        \begin{enumerate}
            \item[i)] $\sup A_{\mu} < \delta_{\xi} < \delta_{\xi} + \omega < \min A_{\xi}$, for every $\omega \leq \mu < \xi$,
            \item[ii)] $\left( x_{g_{\xi, \beta}(n)} \right)_{n \in \omega}$ have $p_{\xi}-$limit in $\langle \lbrace x_{m, \theta} \mid m \in \mathbb{N}, \theta \in I_{\xi} \cup A_{\xi} \rbrace \rangle$, for every $\beta < \gamma$ and $\xi \in 2^{\mathfrak{c}} \setminus \omega$,
            \item[iii)] $\left( x_{1, \delta_{\xi}+n} \right)_{n \in \omega}$ have no accumulation point in $\langle \lbrace x_{m, \theta} \mid m \in \mathbb{N}, \theta \in \delta_{\xi} + \omega \rbrace \rangle$, for every $\xi \in 2^{\mathfrak{c}} \setminus \omega$.
        \end{enumerate}
    \end{prop}

    \begin{proof}
    Let $\lbrace g_{\xi, \gamma} \mid \omega \leq \xi < 2^{\mathfrak{c}}, \: \omega \leq \gamma < 2^{\mathfrak{c}} \rbrace$ be a family of injective sequences of $\displaystyle\bigcup_{\theta \in [\omega, 2^{\mathfrak{c}})} \left( \mathbb{Q}^{(2^{\mathfrak{c}})} \right)^{\theta}$ such that $I_{\xi} = \displaystyle\bigcup \lbrace \supp{g_{\xi, \beta}(n)} \mid \beta < \gamma, \: n \in \omega \rbrace \subset \xi$.

    Consider the linear subspace $\langle \lbrace [ g_{\omega, \beta} ]_{p_{\omega}} \mid \beta < \gamma_{\omega}  \rbrace \cup \lbrace [ \chi_{\overrightarrow{\mu}} ]_{p_{\omega}} \mid \mu \in I_{\omega} \rbrace \rangle$. So, exists $B_{\omega} \subset \gamma_{\omega}$ such that $\lbrace [ g_{\omega, \beta} ]_{p_{\omega}} \mid \beta \in B_{\omega}  \rbrace \cup \lbrace [ \chi_{\overrightarrow{\mu}} ]_{p_{\omega}} \mid \mu \in I_{\omega} \rbrace$ is a basis. Put $\delta_{\omega} = \omega$ and take $A_{\omega} \subset 2^{\mathfrak c} \setminus  (\delta_{\omega} + \omega)$ such that $\lbrace f_{\theta} \mid \theta \in A_{\omega} \rbrace$ be a indexation of $\lbrace g_{\omega, \beta} \mid \beta \in B_{\omega} \rbrace$. Thereby, $\lbrace [f_{\theta}]_{p_\omega} \mid \theta \in A_{\omega} \rbrace \cup \lbrace [ \chi_{\overrightarrow{\mu}} ]_{p_{\omega}} \mid \mu \in 2^{\mathfrak c} \rbrace$ is linearly independent. 

    Given $\xi \in [\omega, 2^{\mathfrak{c}})$, assume that we have $\lbrace A_{\mu} \mid \omega \leq \mu < \xi \rbrace$ and $\lbrace \delta_{\mu} \mid \omega \leq \mu < \xi \rbrace$. Let $\delta_{\xi} \in 2^{\mathfrak{c}}$ such that $\displaystyle\bigcup_{\mu \in [\omega, \xi)} A_{\mu} \subset \delta_{\xi}$. Analogously, exists $B_{\xi} \subset \gamma_{\xi}$ such that $\lbrace [ g_{\xi, \beta} ]_{p_{\xi}} \mid \beta \in B_{\xi}  \rbrace \cup \lbrace [ \chi_{\overrightarrow{\mu}} ]_{p_{\xi}} \mid \mu \in I_{\xi} \rbrace$ is a basis of the  linear subspace $\langle \lbrace [ g_{\xi, \beta} ]_{p_{\xi}} \mid \beta < \gamma_{\xi}  \rbrace \cup \lbrace [ \chi_{\overrightarrow{\mu}} ]_{p_{\xi}} \mid \mu \in I_{\xi} \rbrace \rangle$. Take $A_{\xi} \subset 2^{\mathfrak c} \setminus  ( \delta_{\xi} + \omega )$ such that $\lbrace f_{\theta} \mid \theta \in A_{\xi} \rbrace$ be a indexation of $\lbrace g_{\xi, \beta} \mid \beta \in B_{\xi}  \rbrace$. Thereby, $\lbrace [f_{\theta}]_{p_\xi} \mid \theta \in A_{\xi} \rbrace \cup \lbrace [ \chi_{\overrightarrow{\mu}} ]_{p_{\xi}} \mid \mu \in 2^{\mathfrak c} \rbrace$ is linearly independent.

    Due to Theorem \ref{homo}, for every $d \in \mathbb{Q}^{(2^{\mathfrak{c}})}$ non-zero, exists a homomorphism $\phi_{d} : \mathbb{Q}^{(2^{\mathfrak{c}})} \to \mathbb{T}$ such that $\phi_d(d) \neq 0$ and $p_{\xi}- \lim\limits_{n \to \infty} \phi_d \left( \frac{1}{N} f_{\theta}(n) \right) = \phi_d \left( \frac{1}{N} \chi_{\theta} \right)$, for each $N \in \mathbb{N}$, $\theta \in A_{\xi}$ and $\xi \in [\omega, 2^{\mathfrak{c}})$. 
    
    Furthermore, for every $\xi \in 2^{\mathfrak{c}} \setminus \omega$ and $m \in \omega$, we can define a homomorphism $\psi_{\xi, m} : \mathbb{Q}^{(2^{\mathfrak{c}})} \to \mathbb{T}$ such that $\left\lbrace \beta < \delta_{\xi}+\omega \mid \psi_{\xi, m} \left( \chi_{\beta} \right) =  [ \pi ] \right\rbrace = [\delta_{\xi}+m, \delta_{\xi} + \omega)$, $\psi_{\xi, m} \mid_{\mathbb{Q}^{(\delta_{\xi}+m)}} = 0$,  and $p_{\xi}- \lim\limits_{n \to \infty} \psi_{\xi, m} \left( \frac{1}{N}  f_{\beta}(n) \right) = \psi_{\xi, m} \left( \frac{1}{N} \chi_{\beta} \right)$, for each $N \in \mathbb{N}$, $\beta \in A_{\xi}$ and $\xi \in [\omega, 2^{\mathfrak{c}})$. 

    Fix an enumeration $\lbrace h_{\beta} \mid \beta \in 2^{\mathfrak{c}} \rbrace$ of $\lbrace \phi_{d} \mid d \in \mathbb{Q}^{(2^{\mathfrak{c}})}  \setminus \lbrace 0 \rbrace \rbrace \cup \lbrace \psi_{\xi, m} \mid \xi \in 2^{\mathfrak{c}} \setminus \omega, m \in \omega \rbrace$. We define the monomorphism $\Phi : \mathbb{Q}^{(2^{\mathfrak{c}})} \to \mathbb{T}^{2^{\mathfrak{c}}}$ by $\Phi(d) = (h_{\beta}(d))_{\beta \in 2^{\mathfrak{c}}}$. 
    
    For each $N \in \mathbb{N}$ and $\xi \in 2^{\mathfrak{c}}$, let $x_{N, \xi} = \Phi \left( \frac{1}{N} \chi_{\xi} \right)$. In that way, is clear that $\lbrace x_{N, \xi} \mid N \in \mathbb{N}, \xi \in 2^{\mathfrak{c}} \rbrace$ is linearly independent. 
    
    Let $\xi \in [\omega, 2^{\mathfrak{c}})$. We know that $p_{\xi}- \lim\limits_{k \to \infty} \Phi \left( \frac{1}{N} f_{\theta}(k) \right) = \Phi \left( \frac{1}{N} \chi_{\theta} \right) = x_{N, \theta}$, for every $N \in \mathbb{N}$ and $\theta \in A_{\xi}$. As a consequence, due to Proposition \ref{plim}, $\left( x_{g_{\xi, \beta}(n)} \right)_{n \in \omega}$ have $p_{\xi}-$limit in $\langle \lbrace x_{m, \theta} \mid m \in \mathbb{N}, \theta \in I_{\xi} \cup A_{\xi} \rbrace \rangle$, for every $\beta < \gamma$. 
    
    On the other hand, for every $t \in \mathbb{Q}^{ (\delta_{\xi}+\omega) }$, exists $n_0 \in \omega$ such that $\supp{t} \subset \delta_{\xi} + n_0$. Let $\lambda \in 2^{\mathfrak{c}}$ such that $\psi_{\delta_{\xi}, n_0} = h_{\lambda}$. Thereby, $h_{\lambda} \left( \chi_{\delta_{\xi+m}} \right) = [\pi] $, if $m \geq n_0$, and $h_{\lambda} \left( \chi_{\delta_{\xi}+m} \right) = 0$, if $m < n_o$. As a consequence, $x_t(\lambda)$ is not an accumulation point of $\left( x_{1, \delta_{\xi} + n}(\lambda) \right)_{n \in \omega}$. That is, $\left( x_{1, \delta_{\xi} + n} \right)_{n \in \omega}$ have no accumulation point in $\langle \lbrace x_{m, \alpha} \mid m \in \mathbb{N}, \alpha \in \delta_{\xi}+\omega  \rbrace \rangle$.
    \end{proof}

    \begin{lem} \label{Lema2c}
    Let $\lbrace p_{\xi} \mid \omega \leq \xi < 2^{\mathfrak{c}} \rbrace$ be a family of incomparable selective ultrafilters and $\lbrace g_{\xi, \gamma} \mid \omega \leq \xi < 2^{\mathfrak{c}}, \: \omega \leq \gamma < 2^{\mathfrak{c}} \rbrace$ a enumeration of the injective sequences of $\displaystyle\bigcup_{\theta \in [\omega, 2^{\mathfrak{c}})} \left( \mathbb{Q}^{(2^{\mathfrak{c}})} \right)^{\theta}$ such that
        \begin{itemize}
            \item $I_{\xi} = \displaystyle\bigcup \lbrace \supp{g_{\xi, \beta}(n)} \mid \beta < \gamma, \: n \in \omega \rbrace \subset \xi$,
            \item each $g_{\xi}$ appears $2^{\mathfrak{c}}$ times in the enumeration.
        \end{itemize}
    Consider $\lbrace \delta_{\xi} \mid \omega \leq \xi < 2^{\mathfrak{c}} \rbrace$, $\lbrace A_{\xi} \mid \omega \leq \xi < 2^{\mathfrak{c}} \rbrace$ and $\lbrace x_{m, \beta} \mid m \in \mathbb{N}, \beta \in 2^{\mathfrak{c}} \rbrace$ as in Proposition \ref{Prop2c}. Given a enumeration $\lbrace q_{\eta} \mid \eta < 2^{\mathfrak{c}} \rbrace$ of the free ultrafilters in $\omega$, there exists increasing sequences $\lbrace J_{\xi} \mid \xi < 2^{\mathfrak{c}} \rbrace$, $\lbrace S_{\xi} \mid \xi < 2^{\mathfrak{c}} \rbrace$ in $[2^{\mathfrak{c}}]^{<2^{\mathfrak{c}}}$ and $\lbrace \mu_{\xi} \mid \xi < 2^{\mathfrak{c}} \rbrace$ in $ 2^{\mathfrak{c}} \setminus \omega$ satisfying
        \begin{enumerate}
            \item[i)] $J_{\xi} \cap S_{\xi} = \emptyset$,
            \item[ii)] $g_{\mu_{\xi}} = g_{\delta_{\xi}}$, where $\delta_{\xi}$ is the least ordinal $\delta$ in $2^{\mathfrak{c}} \setminus \omega$ such that $I_{\delta} \subset J_{\xi}$ and $g_{\delta} \neq g_{\mu_{\eta}}$, for every $\eta < \xi$,
            \item[iii)] $\left( x_{g_{\mu_{\xi}, \beta}(n)} \right)_{n \in \omega}$ have $p_{\mu_{\xi}}-$limit in $\langle \lbrace x_{m, \theta} \mid m \in \mathbb{N},\theta \in J_{\xi+1} \rbrace  \rangle$, for every $\beta < \gamma$,
            \item[iv)] $\langle \lbrace x_{m, \theta} \mid m \in \mathbb{N}, \theta \in J_{\xi} \cup I \rbrace  \rangle$ is not $q_{\xi}-$compact, for each $I \subset 2^{\mathfrak{c}} \setminus S_{\xi}$.
        \end{enumerate}
    \end{lem}

    \begin{proof}
    Let $\lbrace q_{\eta} \mid \eta < 2^{\mathfrak{c}} \rbrace$ be a enumeration of $\omega^*$. We know that $\left( x_{1, \delta_{\omega} + n} \right)_{n \in \omega}$ have no accumulation point in the group generated by $\lbrace x_{m, \alpha} \mid m \in \mathbb{N}, \alpha \in \delta_{\omega} + \omega \rbrace$. 
    
    We claim that exists $\theta_0 \in 2^{\mathfrak{c}} \setminus (\delta_{\omega} + \omega)$ such that $\left( x_{1, \delta_{\omega} + n} \right)_{n \in \omega}$ don't have $q_0-$limit in $\langle \lbrace x_{m, \alpha} \mid m \in \mathbb{N}, \alpha \in 2^{\mathfrak{c}} \setminus \lbrace \theta_0 \rbrace \rbrace \rangle$. In fact, if $\left( x_{1, \delta_{\omega} + n} \right)_{n \in \omega}$ have $q_0-$limit in $\langle \lbrace x_{m, \alpha} \mid m \in \mathbb{N}, \alpha \in 2^{\mathfrak{c}} \rbrace \rangle$, exists $d_0 \in \mathbb{Q}^{(2^{\mathfrak{c}})}$ such that $x_{d_0}$ is that $q_0-$limit. From there we can choose $\theta_0 \in \supp{d_0} \setminus (\delta_{\omega} + \omega)$. Otherwise, we can take any $\theta_0 \in 2^{\mathfrak{c}} \setminus (\delta_{\omega} + \omega)$.
    
    We define $J_{0} = \omega \cup [\delta_{\omega}, \delta_{\omega} + \omega )$ and $S_{0} = \lbrace \theta_{0} \rbrace$. Let $\gamma_0 \in 2^{\mathfrak{c}} \setminus \omega$ be minimal such that $I_{\gamma_0} \subset J_0$. So, exists $\mu_0 \in 2^{\mathfrak{c}} \setminus \omega$ such that $g_{\mu_0} = g_{\gamma_0}$ and $\theta_0 < \min A_{\mu_0}$. Put $J_1 = J_0 \cup A_{\mu_0}$, clearly $(i), (ii)$, $(iii)$ are satisfied. 
    
    Given $\xi < 2^{\mathfrak{c}}$, assume that we have $\lbrace J_{\beta} \mid \beta < \xi \rbrace$, $\lbrace S_{\beta} \mid \beta < \xi \rbrace$ and $\lbrace \mu_{\beta} \mid \beta < \xi \rbrace$. Let $L = \displaystyle\bigcup_{\beta < \xi} J_{\beta}$ and $T = \displaystyle\bigcup_{\beta < \xi} S_{\beta}$, notice that are disjoint and $L, T \in [2^{\mathfrak{c}}]^{< 2^{\mathfrak{c}}}$. So, exists $\beta_{\xi} \in 2^{\mathfrak{c}} \setminus \omega$ such that $\sup (L \cup T) < \delta_{\beta_{\xi}}$.
    
    Due to item $(iii)$ of Proposition \ref{Prop2c}, we know that $\left( x_{1, \delta_{\beta_{\xi}} + n} \right)_{n \in \omega}$ have no  accumulation point in the group generated by $\lbrace x_{m, \alpha} \mid m \in \mathbb{N}, \alpha \in \delta_{\beta_{\xi}} + \omega \rbrace$. Analogously, exists $\theta_{\xi} \in 2^{\mathfrak{c}} \setminus (\delta_{\beta_{\xi}} + \omega)$ such that $\left( x_{1, \delta_{\beta_{\xi}} + n} \right)_{n \in \omega}$ don't have $q_{\xi}-$limit in the group generated by $\lbrace x_{m, \alpha} \mid m \in \mathbb{N}, \alpha \in 2^{\mathfrak{c}} \setminus \lbrace \theta_{\xi} \rbrace \rbrace$.
    
    We define $J_{\xi} = L \cup [\delta_{\beta_{\xi}}, \delta_{\beta_{\xi}} + \omega)$ and $S_{\xi} = T \cup \lbrace \theta_{\xi} \rbrace$. Now, let $\gamma_{\xi} \in 2^{\mathfrak{c}} \setminus \omega$ be minimal such that $I_{\gamma_{\xi}} \subset J_{\xi}$ and $g_{\gamma_{\xi}} \neq g_{\mu_{\eta}}$, for every $\eta < \xi$. So, exists $\mu_{\xi} \in 2^{\mathfrak{c}} \setminus \omega$ such that $g_{\mu_{\xi}} = g_{\gamma_{\xi}}$ and $\theta_{\xi} < \min A_{\mu_{\xi}}$. Put $J_{\xi +1} = J_{\xi} \cup A_{\mu_{\xi}}$, we obtain $(i)$, $(ii)$ and $(iii)$. 

    Finally, about $(iv)$, given $I \subseteq 2^{\mathfrak{c}} \setminus S_{\xi}$, clearly $I \cup J_{\xi} \subseteq 2^{\mathfrak{c}} \setminus S_{\xi}  \subseteq 2^{\mathfrak{c}} \setminus \lbrace \theta_{\xi} \rbrace $. As a consequence, $\left( x_{1, \delta_{\beta_{\xi}} + n} \right)_{n \in \omega}$ don't have $q_{\xi}-$limit in the group generated by $\lbrace x_{m, \alpha} \mid m \in \mathbb{N}, \alpha \in J_{\xi} \cup I \rbrace$. That is,  $\langle \lbrace x_{m, \alpha} \mid m \in \mathbb{N}, \alpha \in J_{\xi} \cup I \rbrace \rangle$ is not $q_{\xi}-$compact, for each $I \subset 2^{\mathfrak{c}} \setminus S_{\xi}$.
    \end{proof}
    
    Below we present a solution to the Problem 6.3 of \cite{Boero} for the case $\alpha = 2^{\mathfrak c}$.

    \begin{teo} \label{Caso2c}
    There exists a torsion-free Abelian topological group $G$ such that $G^{\gamma}$ is countably compact, for every $\gamma < 2^{\mathfrak{c}}$, but $G^{2^{\mathfrak{c}}}$ is not countably compact.
    \end{teo}

    \begin{proof}
    Let $\lbrace p_{\xi} \mid \omega \leq \xi < 2^{\mathfrak{c}} \rbrace$, $\lbrace g_{\xi, \gamma} \mid \omega \leq \xi < 2^{\mathfrak{c}}, \: \omega \leq \gamma < 2^{\mathfrak{c}} \rbrace$,  $\lbrace \mu_{\xi} \mid \xi < 2^{\mathfrak{c}} \rbrace$, $\lbrace x_{\xi} \mid \xi < 2^{\mathfrak{c}} \rbrace$, $\lbrace q_{\xi} \mid \xi < 2^{\mathfrak{c}} \rbrace$, $\lbrace J_{\xi} \mid \xi < 2^{\mathfrak{c}} \rbrace$, $\lbrace S_{\xi} \mid \xi < 2^{\mathfrak{c}} \rbrace$ be as in Lemma \ref{Lema2c}. For every $\xi < 2^{\mathfrak{c}}$, we denote by $G_{\xi }$ the group generated by $\left\lbrace x_{m, \alpha} \mid m \in \mathbb{N}, \alpha \in J_{\xi } \right\rbrace$. Let $J = \displaystyle\bigcup_{\xi < 2^{\mathfrak{c}}} J_{\xi}$ and $S = \displaystyle\bigcup_{\xi < 2^{\mathfrak{c}}} S_{\xi}$.
    
    We define $G$ as the group generated by $\lbrace x_{m, \eta} \mid m \in \mathbb{N}, \eta \in J \rbrace$. Notice that $G = \displaystyle\bigcup_{\xi < 2^{\mathfrak{c}}} G_{\xi}$. Furthermore, for each $d \in \mathbb{Q}^{(J)} \setminus \lbrace 0 \rbrace$ and $m \in \mathbb{Z} \setminus \lbrace 0 \rbrace$, we have $m.x_{d} \neq 0$. It follows that $G$ is a torsion-free topological group.
    
    We will show that $G^{\alpha}$ is countably compact, for every $\alpha < 2^{\mathfrak{c}}$. In fact, let $\alpha \in [ \omega, 2^{\mathfrak{c}})$. Given an injective sequence $(y_n)_{n \in \omega} \subset G^{\alpha}$, exists $\theta < 2^{\mathfrak{c}}$ such that $(y_n)_{n \in \omega} = \left( x_{g_{\theta, \beta}(n)} \right)_{n \in \omega, \beta < \alpha}$. Whereas $I_{\theta} \subset J$ and $2^{\mathfrak{c}}$ is regular, exists $\eta < 2^{\mathfrak{c}}$ such that $g_{\theta, \alpha} = g_{\mu_{\eta}, \alpha}$. Due to item $(iii)$ of Lemma \ref{Lema2c}, we know that $\left( x_{g_{\mu_{\eta}, \beta}(n)} \right)_{n \in \omega}$ have $p_{\mu_{\eta}}-$limit in $G_{\eta+1} \subset G$, for every $\beta < \alpha$. This is, $(y_n)_{n \in \omega}  = \left( x_{g_{\mu_{\eta}, \beta}(n)} \right)_{n \in \omega, \beta < \alpha}$ have $p_{\mu_{\eta}}-$limit in $G^{\alpha}$. Thus, every sequence of $G^{\alpha}$ have an accumulation point.
    
    On the other hand, we claim that $G^{2^{\mathfrak{c}}}$ is not countably compact. Indeed, because of item $(iv)$ of Lemma \ref{Lema2c}, we have that $G$ is not $q_{\xi}-$compact, for every $\xi < 2^{\mathfrak{c}}$. If $G^{2^{\mathfrak{c}}}$ is countably compact, then exists $\gamma \in 2^{\mathfrak{c}}$ such that $G$ is $q_{\gamma}-$compact, due to Ginsburg and Saks' theorem, but that is false. Therefore, $G^{2^{\mathfrak{c}}}$ is not countably compact. 
    \end{proof}

    \subsection{Case $\alpha < 2^{\mathfrak{c}}$}

    \begin{prop} \label{Propalpha}
    Let $\alpha < 2^{\mathfrak{c}}$ be an infinite cardinal and $\lbrace p_{\xi} \mid \omega \leq \xi < 2^{\mathfrak{c}} \rbrace$ a family of incomparable selective ultrafilters.

    Given a family $\lbrace g_{\xi, \gamma} \mid \omega \leq \xi < 2^{\mathfrak{c}}, \: \omega \leq \gamma < \alpha \rbrace$ of injective sequences of $\displaystyle\bigcup_{\theta \in [\omega, \alpha)} \left( \mathbb{Q}^{(2^{\mathfrak{c}})} \right)^{\theta}$ such that $I_{\xi} = \displaystyle\bigcup \lbrace \supp{g_{\xi, \beta}(n)} \mid \beta < \gamma, \: n \in \omega \rbrace \subset \xi$, \footnote{Here $g_{\xi, \gamma} = ( g_{\xi, \beta} )_{\beta < \gamma}$ is an injective sequence in $\left( \mathbb{Q}^{(2^{\mathfrak{c}})} \right)^{\gamma}$.  } there exists a family $\lbrace A_{\xi} \mid \omega \leq \xi < 2^{\mathfrak{c}} \rbrace$ of disjoint subsets of $2^{\mathfrak{c}} \setminus \alpha$, with $\vert A_{\xi} \vert < \alpha$, and a linearly independent subset $\lbrace x_{m, \beta} \mid m \in \mathbb{N}, \beta \in 2^{\mathfrak{c}} \rbrace$ of $\mathbb{T}^{2^{\mathfrak{c}}}$ such that
        \begin{enumerate}
            \item[i)] $\left( x_{g_{\xi, \beta}(n)} \right)_{n \in \omega}$ have $p_{\xi}-$limit in $\langle \lbrace x_{m, \theta} \mid m \in \mathbb{N}, \theta \in I_{\xi} \cup A_{\xi} \rbrace \rangle$, for each $\beta < \gamma$ and $\xi \in 2^{\mathfrak{c}} \setminus \omega$, 
            \item[ii)] for every $A \subset \alpha$, exists $\beta \in 2^{\mathfrak{c}}$ such that $\lbrace \theta \in \alpha \mid x_{1, \theta}(\beta) = [\pi] \rbrace = A$ and $x_{m, \theta}(\beta) = 0$, for each $m \in \mathbb{N}$ and $\theta \in \alpha \setminus A$,
            \item[iii)] if $p, q \in \omega^*$ are distinct, then for each limit ordinals $\beta_0, \gamma_0, \beta_1, \gamma_1 \in \alpha$, with $\beta_i \neq \gamma_i$ for all $i < 2$, we have
                \begin{itemize}
                    \item $p- \lim\limits_{n \to \omega} x_{1, \beta_0 + n} + x_{1, \gamma_0 + n} \neq q- \lim\limits_{n \to \omega} x_{1, \beta_1 + n} + x_{1, \gamma_1 + n}$,
                    \item $p- \lim\limits_{n \to \omega} x_{1, \beta + n} \neq p- \lim\limits_{n \to \omega} x_{1, \gamma + n}$, if $\beta \neq \gamma$. 
                \end{itemize}
            \end{enumerate}
    \end{prop}

    \begin{proof}
    Let $\lbrace g_{\xi, \gamma} \mid \omega \leq \xi < 2^{\mathfrak{c}}, \: \omega \leq \gamma < \alpha \rbrace$ be a family of injective sequences of $\displaystyle\bigcup_{\theta \in [\omega, \alpha)} \left( \mathbb{Q}^{(2^{\mathfrak{c}})} \right)^{\theta}$ such that $I_{\xi} = \displaystyle\bigcup \lbrace \supp{g_{\xi, \beta}(n)} \mid \beta < \gamma, \: n \in \omega \rbrace \subset \xi$.
    
    Consider the linear subspace $\langle \lbrace [ g_{\omega, \beta} ]_{p_{\omega}} \mid \beta < \gamma_{\omega}  \rbrace \cup \lbrace [ \chi_{\overrightarrow{\mu}} ]_{p_{\omega}} \mid \mu \in I_{\omega} \rbrace \rangle$. So, exists $B_{\omega} \subset \gamma_{\omega}$ such that $\lbrace [ g_{\omega, \beta} ]_{p_{\omega}} \mid \beta \in B_{\omega}  \rbrace \cup \lbrace [ \chi_{\overrightarrow{\mu}} ]_{p_{\omega}} \mid \mu \in I_{\omega} \rbrace$ is a basis. Take $A_{\omega} \subset 2^{\mathfrak c} \setminus \alpha$ such that $\lbrace f_{\theta} \mid \theta \in A_{\omega} \rbrace$ be a indexation of $\lbrace g_{\omega, \beta} \mid \beta \in B_{\omega} \rbrace$. Thereby, $\lbrace [f_{\theta}]_{p_\omega} \mid \theta \in A_{\omega} \rbrace \cup \lbrace [ \chi_{\overrightarrow{\mu}} ]_{p_{\omega}} \mid \mu \in 2^{\mathfrak c} \rbrace$ is linearly independent. 

    Given $\xi \in [\omega, 2^{\mathfrak{c}})$, assume that we have $\lbrace A_{\mu} \mid \omega \leq \mu < \xi \rbrace$. Let $\delta_{\xi} \in 2^{\mathfrak{c}}$ such that $\displaystyle\bigcup_{\mu \in [\omega, \xi)} A_{\mu} \subset \delta_{\xi}$. In the same way, exists $B_{\xi} \subset \gamma_{\xi}$ such that $\lbrace [ g_{\xi, \beta} ]_{p_{\xi}} \mid \beta \in B_{\xi}  \rbrace \cup \lbrace [ \chi_{\overrightarrow{\mu}} ]_{p_{\xi}} \mid \mu \in I_{\xi} \rbrace$ is a basis of the  linear subspace $\langle \lbrace [ g_{\xi, \beta} ]_{p_{\xi}} \mid \beta < \gamma_{\xi}  \rbrace \cup \lbrace [ \chi_{\overrightarrow{\mu}} ]_{p_{\xi}} \mid \mu \in I_{\xi} \rbrace \rangle$. Take $A_{\xi} \subset 2^{\mathfrak c} \setminus \delta_{\xi}$ such that $\lbrace f_{\theta} \mid \theta \in A_{\xi} \rbrace$ be a indexation of $\lbrace g_{\xi, \beta} \mid \beta \in B_{\xi}  \rbrace$. Thereby, $\lbrace [f_{\theta}]_{p_\xi} \mid \theta \in A_{\xi} \rbrace \cup \lbrace [ \chi_{\overrightarrow{\mu}} ]_{p_{\xi}} \mid \mu \in 2^{\mathfrak c} \rbrace$ is linearly independent.
    
    In virtue of Theorem \ref{homo}, for every $d \in \mathbb{Q}^{(2^{\mathfrak{c}})}$ non-zero, exists a homomorphism $\phi_{d} : \mathbb{Q}^{(2^{\mathfrak{c}})} \to \mathbb{T}$ such that $\phi_d(d) \neq 0$ and $p_{\xi}- \lim\limits_{n \to \infty} \phi_d \left( \frac{1}{N} f_{\theta}(n) \right) = \phi_d \left( \frac{1}{N} \chi_{\theta} \right)$, for each $N \in \mathbb{N}$, $\theta \in A_{\xi}$ and $\xi \in [\omega, 2^{\mathfrak{c}})$.
    
    Furthermore, for every $A \subset \alpha$, exists a homomorphism $\psi_A : \mathbb{Q}^{(2^{\mathfrak{c}})} \to \mathbb{T}$ such that $\lbrace \beta \in \alpha \mid \psi_A \left( \chi_{\beta} \right) = [\pi] \rbrace = A$, $\psi_A \mid_{\mathbb{Q}^{(\alpha \setminus A)}} = 0$, and $p_{\xi}- \lim\limits_{n \to \infty} \psi_A \left( \frac{1}{N} f_{\theta}(n) \right) = \psi_A \left( \frac{1}{N} \chi_{\theta} \right)$, for each $N \in \mathbb{N}$, $\theta \in A_{\xi}$ and $\xi \in [\omega, 2^{\mathfrak{c}})$.

    Fix an enumeration $\lbrace h_{\beta} \mid \beta \in 2^{\mathfrak{c}} \rbrace$ of $\lbrace \phi_{d} \mid d \in \mathbb{Q}^{(2^{\mathfrak{c}})}  \setminus \lbrace 0 \rbrace \rbrace \cup \lbrace \psi_{A} \mid A \subset \alpha \rbrace$. We define a monomorphism $\Phi : \mathbb{Q}^{(2^{\mathfrak{c}})} \to \mathbb{T}^{2^{\mathfrak{c}}}$ by $\Phi(f) = (h_{\beta}(f))_{\beta \in 2^{\mathfrak{c}}}$. For every $N \in \mathbb{N}$ and $\xi \in 2^{\mathfrak{c}}$, let $x_{N, \xi} = \Phi \left( \frac{1}{N} \chi_{\xi} \right)$. Is clear that $\lbrace x_{N, \xi} \mid N \in \mathbb{N}, \xi \in 2^{\mathfrak{c}} \rbrace$ is linearly independent.
    
    Let $\xi \in [\omega, 2^{\mathfrak{c}})$. We know that $p_{\xi}- \lim\limits_{k \to \infty} \Phi \left( \frac{1}{N} f_{\beta}(k) \right) = \Phi \left( \frac{1}{N} \chi_{\beta} \right) = x_{N, \beta}$, for every $N \in \mathbb{N}$ and $\beta \in A_{\xi}$. Due to Proposition \ref{plim}, $\left( x_{g_{\xi, \beta}(n)} \right)_{n \in \omega}$ have $p_{\xi}-$limit in $\langle \lbrace x_{m, \theta} \mid m \in \mathbb{N},\theta \in I_{\xi} \cup A_{\xi} \rbrace  \rangle$, for every $\beta < \gamma$. This is, the item $(i)$ is true.
    
    About item $(ii)$, given $A \subset \alpha$, exists $\beta \in 2^{\mathfrak{c}}$ such that $\psi_A = h_{\beta}$. Then $\lbrace \theta \in \alpha \mid x_{1, \theta}(\beta) = [\pi] \rbrace = A$ and $x_{m, \theta}(\beta) = 0$, for each $m \in \mathbb{N}$ and $\theta \in \alpha \setminus A$. 
    
    Finally, about item $(iii)$, let $p, q \in \omega^*$ distinct. Let $\beta_0, \gamma_0, \beta_1, \gamma_1 \in \alpha$ be limit ordinals, with $\beta_i \neq \gamma_i$ for all $i < 2$. Without lose of generality, consider $\beta_0 < \gamma_0$ and $\beta_1 < \gamma_1$. So, applying $(ii)$, exists $\theta \in 2^{\mathfrak{c}}$ such that $\lbrace \eta \in \alpha \mid x_{1, \eta}(\theta) = [\pi] \rbrace = [0, \beta_0) \cup [\gamma_0 + \omega,\alpha) \cup [\beta_1, \gamma_1)$. As a consequence, $x_{1, \beta_1 + n}(\theta) = [\pi]$ and $x_{1, \beta_0 + n}(\theta) = x_{1, \gamma_0 + n}(\theta)  = x_{1, \gamma_1 + n}(\theta) = 0$, for every $n \in \omega$. That is, $p- \lim\limits_{n \to \omega} x_{1, \beta_0 + n} + x_{1, \gamma_0 + n} \neq q- \lim\limits_{n \to \omega} x_{1, \beta_1 + n} + x_{1, \gamma_1 + n}$.
    
    If $\beta \neq \gamma$, exists $\theta \in 2^{\mathfrak{c}}$ such that $\lbrace \eta \in \alpha \mid x_{1, \eta}(\theta) = [\pi] \rbrace = [\beta, \gamma)$. So, $x_{1, \beta + n}(\theta) = [\pi]$ and $x_{1, \gamma + n}(\theta) = 0$, for every $n \in \omega$. Thereby, $p- \lim\limits_{n \to \omega} x_{1, \beta + n} \neq p- \lim\limits_{n \to \omega} x_{1, \gamma + n}$.
    \end{proof}

    \begin{lem} \label{Lemaalpha}
    Let $\alpha < 2^{\mathfrak{c}}$ be an infinite cardinal, $\lbrace p_{\xi} \mid \omega \leq \xi < 2^{\mathfrak{c}} \rbrace$ a family of incomparable selective ultrafilters and $\lbrace g_{\xi, \gamma} \mid \omega \leq \xi < 2^{\mathfrak{c}}, \: \omega \leq \gamma < \alpha \rbrace$ an enumeration of the injective sequences of $\displaystyle\bigcup_{\theta \in [\omega, \alpha)} \left( \mathbb{Q}^{(2^{\mathfrak{c}})} \right)^{\theta}$ such that
        \begin{itemize}
            \item $I_{\xi} = \displaystyle\bigcup \lbrace \supp{g_{\xi, \beta}(n)} \mid \beta < \gamma, \: n \in \omega \rbrace \subset \xi$,
            \item each $g_{\xi}$ appears $2^{\mathfrak{c}}$ times in the enumeration.
        \end{itemize}
    Consider $\lbrace A_{\xi} \mid \omega \leq \xi < 2^{\mathfrak{c}} \rbrace$, $\lbrace x_{m, \beta} \mid m \in \mathbb{N}, \beta \in 2^{\mathfrak{c}} \rbrace$ as in Proposition \ref{Propalpha} and let $\mathcal{P} = \lbrace p \in \omega^{*} \mid p- \lim\limits_{n \to \infty} x_{1, \theta + n} \in \langle \lbrace x_{m, \eta} \mid m \in \mathbb{N}, \eta \in 2^{\mathfrak{c}}  \rbrace  \rangle, \textrm{for every } \theta < \alpha \textrm{ limit} \rbrace$. Then, there exists increasing sequences $\lbrace K_{\xi} \mid \xi \in 2^{\mathfrak{c}} \rbrace \subset [2^{\mathfrak{c}}]^{< 2^{\mathfrak{c}}}$, $\lbrace P_{\xi} \mid \xi \in 2^{\mathfrak{c}} \rbrace \subset \mathcal{P}$, $\lbrace S_{\xi} \mid \xi \in 2^{\mathfrak{c}}  \rbrace \subset [2^{\mathfrak{c}} \setminus \alpha]^{< 2^{\mathfrak{c}}}$ and $\lbrace \lambda_{\xi} \mid \xi \in 2^{\mathfrak{c}} \rbrace \subset 2^{\mathfrak{c}} \setminus \omega$ satisfying
        \begin{enumerate}
            \item[i)] $p \in P_{\xi}$ if, and only if, $p \in \mathcal P$  and exists $\beta_1 < \beta_2 < \alpha$ limits such that $$p- \lim\limits_{n \to \infty} x_{1, \beta_1+n} - x_{1, \beta_2+n} \in \left\langle \left\lbrace x_{m, \eta} \mid m \in \mathbb{N}, \eta \in K_{\xi} \right\rbrace \right\rangle,$$
            \item[ii)] For every $p \in P_{\xi}$, exists $\beta < \alpha$ limit such that $$p- \lim\limits_{n \to \infty} x_{1, \beta + n} \notin  \left\langle  \lbrace x_{m, \eta} \mid m \in \mathbb{N}, \eta \in 2^{\mathfrak{c}} \setminus S_{\xi} \rbrace \right\rangle,$$
            \item[iii)] $g_{\lambda_{\xi}} = g_{\delta_{\xi}}$, where $\delta_{\xi} \in 2^{\mathfrak{c}} \setminus \omega$ is minimal such that $I_{\delta_{\xi}} \subset K_{\xi}$ and $g_{\delta_{\xi}} \neq g_{\lambda_{\eta}}$, for every $\eta < \xi$,
            \item[iv)] $p_{\lambda_{\xi}}- \lim\limits_{n \to \infty} x_{g_{\lambda_{\xi}, \beta}(n)} \in \left\langle \left\lbrace x_{m, \eta} \mid m \in \mathbb{N}, \eta \in K_{\xi+1} \right\rbrace \right\rangle$, for every $\beta < \gamma$,
            \item[v)] $\vert K_{\xi} \vert \leq \vert \xi \vert + \alpha$, $\vert P_{\xi} \vert \leq \vert \xi \vert + \alpha$ and $\vert S_{\xi} \vert \leq \vert \xi \vert + \alpha$,
            \item[vi)] $K_{\xi} \cap S_{\xi} = \emptyset$,
            \item[vii)] If $\xi$ is an ordinal limit, we have $P_{\xi} = \displaystyle\bigcup_{\gamma < \xi} P_{\gamma}$ and $S_{\xi} = \displaystyle\bigcup_{\gamma < \xi} S_{\gamma}$.  
        \end{enumerate}
    \end{lem}

    \begin{proof}
    Define $K_0 = \alpha$. We will show that $P_0 = \emptyset$. In fact, let $t \in \mathbb{Q}^{(\alpha)}$. By item $(ii)$ of Proposition \ref{Propalpha}, for each $\beta_1 < \beta_2 < \alpha$ limits, exists $\mu \in 2^{\mathfrak{c}}$ such that $\lbrace \eta < \alpha \mid x_{1, \eta}(\mu) = [\pi]  \rbrace = [\beta_1, \beta_2) \setminus \supp{t}$ and $x_{t}(\mu) = 0$. So, $x_t$ is not an accumulation point of $\left( x_{1, \beta_1 + n} - x_{1, \beta_2+n} \right)_{n \in \omega}$. Put $S_0 = \emptyset$, $\lambda_0 = \omega$ and $K_1 = K_0 \cup A_{\lambda_0}$, we obtain the desire conditions. 
    
    Given $\xi < 2^{\mathfrak{c}}$, assume that we have $\lbrace K_{\beta} \mid \beta < \xi \rbrace$, $\lbrace S_{\beta} \mid \beta < \xi \rbrace$, $\lbrace P_{\theta} \mid \theta < \eta < \xi \rbrace$ and $\lbrace \lambda_{\theta} \mid \theta < \eta < \xi \rbrace$.
    
    In case $\xi$ is not limit, exists $\mu \in 2^{\mathfrak{c}}$ such that $\xi = \mu +1$. Take $\delta_{\mu} \in 2^{\mathfrak{c}} \setminus \omega$ minimal such that $I_{\delta_{\mu}} \subset K_{\mu}$ and $g_{\delta_{\mu}} \neq g_{\lambda_{\eta}}$, for every $\eta < \mu$. So, exists $\lambda_{\mu} \in 2^{\mathfrak{c}} \setminus \omega$ such that $g_{\lambda_{\mu}} = g_{\delta_{\mu}}$ and $A_{\lambda_{\mu}} \cap S_{\mu} = \emptyset$. Due to item $(i)$ of Proposition \ref{Propalpha}, we know that $p_{\lambda_{\mu}}- \lim\limits_{n \to \infty} x_{g_{\lambda_{\mu}, \beta}(n)} \in \left\langle \left\lbrace x_{m, \eta} \mid m \in \mathbb{N}, \eta \in I_{\lambda_{\mu}} \cup A_{\lambda_{\mu}} \right\rbrace \right\rangle$, for every $\beta < \gamma$. We define $K_{\mu +1} = K_{\mu} \cup A_{\lambda_{\mu}} = \alpha \cup \bigcup_{\nu < \mu+1} A_{\lambda_{\nu}}$. As a consequence, the items $(iii)$ and $(iv)$ are satisfied. Furthermore, $\vert K_{\xi} \vert \leq \alpha + \vert \xi \vert$. 
    
    Consider $P_{\xi}$ defined as in $(i)$. For every $p \in P_{\xi}$, there exists $\beta_1 < \beta_2 < \alpha$ such that $x_p = p - \lim\limits_{n \to \omega} x_{1, \beta_1+n} - x_{1, \beta_2+n} \in \left\langle \left\lbrace x_{m, \eta} \mid m \in \mathbb{N}, \eta \in K_{\xi} \right\rbrace \right\rangle$. By the item $(iii)$ of Proposition \ref{Propalpha}, $x_p \neq x_q$, if $p, q \in P_{\xi}$ are distinct. In that way, $\vert P_{\xi} \vert \leq \vert K_{\xi} \vert \leq \alpha + \vert \xi \vert$.
    
    We claim that, given $p \in P_{\mu+1} \setminus P_{\mu}$, exists $\beta < \alpha$ limit such that $p - \lim\limits_{n \to \infty} x_{1, \beta +n} = x_{d_{\beta}} \notin \left\langle \left\lbrace x_{m, \eta} \mid m \in \mathbb{N}, \eta \in K_{\mu+1} \right\rbrace \right\rangle$. In fact, if is not true, then for every $\gamma < \alpha$ limit, there exists $f_{\gamma} \in \mathbb{Q}^{(K_{\mu})}$ and $t_{\gamma} \in \mathbb{Q}^{(A_{\lambda_{\mu}})} \setminus \lbrace 0  \rbrace$ such that $p- \lim\limits_{n \to \infty} x_{1, \gamma + n} = x_{f_{\gamma}} + x_{t_{\gamma}}$. Whereas $\big\vert \mathbb{Q}^{(A_{\lambda_{\mu}})} \big\vert < \alpha$, there exists $\beta_1 < \beta_2 < \alpha$ such that $t_{\beta_1} = t_{\beta_2}$. So, $p - \lim\limits_{n \to \infty} x_{1, \beta_1 +n} - x_{1, \beta_2+n} = x_{f_{\beta_1}} - x_{f_{\beta_2}} \in \left\langle \left\lbrace x_{m, \eta} \mid m \in \mathbb{N}, \eta \in  K_{\mu} \right\rbrace \right\rangle$, but that is false because $p \notin P_{\mu}$.
    
    Now fix $\theta_p \in \supp{d_{\beta}} \setminus K_{\mu+1}$ and define $S_{\mu+1} = S_{\mu} \cup \lbrace \theta_{p} \mid p \in P_{\mu+1} \setminus P_{\mu} \rbrace$. Thereby, $(ii)$, $(v)$ and $(vi)$ are true. 
    
    In case $\xi$ be limit, we define $K_{\xi} = \alpha \cup \bigcup_{\nu < \xi} A_{\lambda_{\nu}}$, $S_{\xi} = \displaystyle\bigcup_{\mu < \xi} S_{\mu}$ and $P_{\xi} = \displaystyle\bigcup_{\mu < \xi} P_{\mu}$. Let $\delta_{\xi} \in 2^{\mathfrak{c}} \setminus \omega$ be minimal such that $I_{\delta_{\xi}} \subset K_{\xi}$ and $g_{\delta_{\xi}} \neq g_{\lambda_{\eta}}$, for every $\eta < \xi$. So, exists $\lambda_{\xi} \in 2^{\mathfrak{c}} \setminus \omega$ such that $g_{\lambda_{\xi}} = g_{\delta_{\xi}}$ and $A_{\lambda_{\xi}} \cap S_{\xi} = \emptyset$. Therefore, making $K_{\xi+1} = K_{\xi} \cup A_{\lambda_{\xi}}$, we obtain $(iii)$, $(iv)$, $(v)$, $(vi)$ and $(vii)$.
    
    Let's see that $(i)$ is true. Given $p \in P_{\xi}$, exists $\mu < \xi$ such that $p \in P_{\mu}$. So, there exists $\beta_1 < \beta_2 < \alpha$ such that $ p- \lim\limits_{n \to \infty} x_{1, \beta_1+n} - x_{1, \beta_2+n} \in \left\langle \left\lbrace x_{m, \eta} \mid m \in \mathbb{N}, \eta \in K_{\mu} \right\rbrace \right\rangle \subset \left\langle \left\lbrace x_{m, \eta} \mid m \in \mathbb{N}, \eta \in K_{\xi} \right\rbrace \right\rangle$. Conversely, let $p \in \mathcal{P}$ such that exists $\gamma_1 < \gamma_2 < \alpha$ such that $ p- \lim\limits_{n \to \infty} x_{1, \gamma_1+n} - x_{1, \gamma_2+n} \in \left\langle \left\lbrace x_{m, \eta} \mid m \in \mathbb{N}, \eta \in K_{\xi} \right\rbrace \right\rangle$. Then, exists $d_{p} \in \mathbb{Q}^{(K_{\xi})}$ such that $x_{d_{p}}$ is that $p-$limit. As a consequence, exists $\gamma < \xi$ such that $\supp{d_{p}} \subset K_{\gamma}$. This is, $p \in P_{\gamma} \subset P_{\xi}$.
    
    Finally, about $(ii)$, let $p \in P_{\xi}$. exists $\gamma < \xi$ such that $p \in P_{\gamma}$. Thereby, exists $\beta < \alpha$ limit such that $p- \lim\limits_{n \to \infty} x_{1, \beta + n} \notin  \left\langle  \lbrace x_{m, \eta} \mid m \in \mathbb{N}, \eta \in 2^{\mathfrak{c}} \setminus S_{\gamma} \rbrace \right\rangle$. Whereas $S_{\gamma} \subset S_{\xi}$, we have $p- \lim\limits_{n \to \infty} x_{1, \beta + n} \notin  \left\langle  \lbrace x_{m, \eta} \mid m \in \mathbb{N}, \eta \in 2^{\mathfrak{c}} \setminus S_{\xi} \rbrace \right\rangle$.  
    \end{proof}
    
    At the end, we present a solution to the Problem 6.3 of \cite{Boero} for every infinite cardinal $\alpha < 2^{\mathfrak c}$.

    \begin{teo} \label{}
    For every infinite cardinal $\alpha < 2^{\mathfrak{c}}$, there exists a torsion-free Abelian topological group $G$ such that $G^{\gamma}$ is countably compact, for each $\gamma < \alpha$, but $G^{\alpha}$ is not countably compact.
    \end{teo}

    \begin{proof}
    Let $\alpha < 2^{\mathfrak{c}}$ be an infinite cardinal and consider $\lbrace p_{\xi} \mid \omega \leq \xi < 2^{\mathfrak{c}} \rbrace$, $\lbrace g_{\xi, \gamma} \mid \omega \leq \xi < 2^{\mathfrak{c}}, \: \omega \leq \gamma < \alpha \rbrace$, $\lbrace x_{N, \xi} \mid N \in \mathbb{N}, \xi \in 2^{\mathfrak{c}} \rbrace$, $\lbrace P_{\xi} \mid \xi \in 2^{\mathfrak{c}} \rbrace$, $\lbrace S_{\xi} \mid \xi \in 2^{\mathfrak{c}} \rbrace$, $\lbrace K_{\xi} \mid \xi \in 2^{\mathfrak{c}} \rbrace$ and $\lbrace \lambda_{\xi} \mid \xi \in 2^{\mathfrak{c}} \rbrace$ as in Lemma \ref{Lemaalpha}. For every $\mu < 2^{\mathfrak{c}}$, let $G_{\mu} = \left\langle \left\lbrace x_{m, \eta} \mid m \in \mathbb{N}, \eta \in  K_{\mu} \right\rbrace \right\rangle$. Set $K = \displaystyle\bigcup_{\xi < 2^{\mathfrak{c}}} K_{\xi}$ and $S = \displaystyle\bigcup_{\xi < 2^{\mathfrak{c}}} S_{\xi}$.
    
    We define $G$ as the group generated by $\lbrace x_{m, \eta} \mid m \in \mathbb{N}, \eta \in K \rbrace$. Note that $G = \displaystyle\bigcup_{\xi < 2^{\mathfrak{c}}} G_{\xi}$ and $G \subset \langle \lbrace x_{m, \eta} \mid m \in \mathbb{N}, \eta \in 2^{\mathfrak{c}} \setminus S_{\mu} \rangle \rbrace$, for every $\mu < 2^{\mathfrak{c}}$.  Furthermore, for each $d \in \mathbb{Q}^{(K)}$ and $m \in \mathbb{Z}$ non-zeros, we have $m.x_{d} \neq 0$. It follows that $G$ is a torsion-free topological group.
    
    We claim that $G^{\gamma}$ is countably compact, for every $\gamma \in [\omega, \alpha)$. In fact, let $\gamma \in [ \omega, \alpha)$. Given an injective sequence $(y_n)_{n \in \omega} \subset G^{\gamma}$, exists $\theta < 2^{\mathfrak{c}}$ such that $(y_n)_{n \in \omega} = \left( x_{g_{\theta, \beta}(n)} \right)_{n \in \omega, \beta < \gamma}$. Whereas $I_{\theta} \subset K$ and $2^{\mathfrak{c}}$ is regular, exists $\xi < 2^{\mathfrak{c}}$ such that $g_{\theta, \gamma} = g_{\lambda_{\xi}, \gamma}$. Due to item $(iv)$ of Lemma \ref{Lemaalpha}, we know that $\left( x_{g_{\lambda_{\xi}, \beta}(n)} \right)_{n \in \omega}$ have $p_{\lambda_{\xi}}-$limit in $G_{\xi+1} \subset G$, for every $\beta < \gamma$. This is, $(y_n)_{n \in \omega} = \left( x_{g_{\lambda_{\xi}, \beta}(n)} \right)_{n \in \omega, \beta < \gamma}$ have $p_{\lambda_{\xi}}-$limit in $G^{\gamma}$. Thus, every sequence of $G^{\gamma}$ have an accumulation point. 
    
    Finally, we will show that $G^{\alpha}$ is not countably compact. Consider the sequence $\left( z_{\beta, n} \right)_{\beta < \alpha, n \in \omega}$, where  $z_{\beta, n} = x_{1, \beta+n}$, if $\beta < \alpha$ is limit, and $z_{\beta, n} = x_{1, \beta}$, if $\beta < \alpha$ is not limit. Suppose that $\left( z_{\beta, n} \right)_{\beta < \alpha, n \in \omega}$ have an accumulation point in $G^{\alpha}$, then exists $p \in \omega^*$ such that $p- \lim\limits_{n \to \infty} x_{1, \beta + n} \in G$, for every $\beta < \alpha$ limit. This is, $p \in \mathcal{P}$. Whereas $2^{\mathfrak{c}}$ is regular, exists $\xi < 2^{\mathfrak{c}}$ such that $p- \lim\limits_{n \to \infty} x_{1, \beta + n} \in G_{\xi}$, for every $\beta < \alpha$ limit. Thereby, $p \in P_{\xi}$. But, because of item $(ii)$ of Lemma \ref{Lemaalpha}, exists $\theta < \alpha$ limit such that $p - \lim\limits_{n \to \omega} x_{1, \theta+n} \notin  \left\langle \lbrace x_{m, \eta} \mid m \in \mathbb{N}, \eta \in 2^{\mathfrak{c}} \setminus S_{\xi} \rbrace \right\rangle$. As a consequence, $p - \lim\limits_{n \to \omega} x_{1, \theta+n} \notin  G$, which is false. Therefore, $\left( z_{\beta, n} \right)_{\beta < \alpha, n \in \omega}$ does not have an accumulation point in $G^{\alpha}$.  
    \end{proof}

\section{A Wallace semigroup whose cube is countably compact}

    We recall that a Wallace semigroup is a both-sided cancellative topological semigroup that is not a topological group.
    
    From now on, we denote by $\mathcal H$ the set of all finite family of sequences $F$ that are strictly monotone in $\mathbb{Q}$ such that $F \cup \{\vec{1}\}$ is linearly independent in $\mathbb{Q}^{\omega}$.

    \begin{prop} \label{PropW3}
    Let $r^1, r^2, r^3 \in \mathbb{Q}_+^{\omega}$. There exist $h : \omega \to \omega$ strictly increasing and $F \in \mathcal H$ such that $(r^1 \circ h)(n) = k_1 + \displaystyle\sum_{f \in F} a_f.f(n)$, $(r^2 \circ h)(n) = k_2 + \displaystyle\sum_{f \in F} b_f.f(n)$ and $(r^3 \circ h)(n) = k_3 + \displaystyle\sum_{f \in F} c_f.f(n)$, for every $n \in \omega$, where $k_1, k_2, k_3 \in \mathbb{Q}_+$ and $a_f, b_f, c_f \in \mathbb{Q}$ for any $f \in F$.
    \end{prop}

    \begin{proof}
    Given $r^1, r^2, r^3 \in \mathbb{Q}_+^{\omega}$, exists $h : \omega \to \omega$ strictly increasing such that every $r^i \circ h$ is either constant or strictly increasing or strictly decreasing. Let $I \subset \lbrace 1, 2 , 3 \rbrace$ be such that $\left\lbrace r^i \circ h \mid i \in I \right\rbrace \cup \lbrace \overrightarrow{1} \rbrace$ is a basis of the subspace $\langle \lbrace r^i \circ h \mid i \in \lbrace 1, 2 , 3 \rbrace \rbrace \cup \lbrace \overrightarrow{1} \rbrace \rangle$. 
    
    If $I = \lbrace 1, 2 , 3 \rbrace$ or $I = \emptyset$, clearly is done. Otherwise, we have the next cases.
    
    Case 1:  $I$ have two elements. Let $i,j \in \{ 1,2,3 \}$ such that $I = \{ i,j \}$ and consider $k \in \{ 1,2,3 \} \setminus I$. Then, exists $q_i, q_j, c_k \in \mathbb{Q}$ such that $(r^k \circ h)_n = q_i (r^i \circ h)_n + q_j (r^j \circ h)_n + c_k$, for every $n \in \omega$. Also, we have the next subcases:  
        \begin{enumerate}
            \item[i)] The sequence $r^k \circ h$ is constant. This implies $q_i = 0$, $q_j = 0$ and $c_k \geq 0$. We take $F = \lbrace r^i \circ h, r^j \circ h \rbrace$. 
            \item[ii)] The sequence $r^k \circ h$ is strictly monotone. If $r^i \circ h$ and $r^j \circ h$ have different kind of monotonicity of $r^k \circ h$, we obtain $q_i < 0$, $q_j < 0$ and $c_k > 0$. Hence, $F = \lbrace r^i \circ h, r^j \circ h \rbrace$. Else, assume that $r^i \circ h$ is the only sequence with  the same kind of monotonicity of $r^k \circ h$. So $q_i > 0$ and then we fix some $d \in \mathbb{Q}_+$ greater than $\frac{-c_k}{q_i}$. Therefore, $dq_i + c_k > 0$ and $F = \lbrace (r^i \circ h) - d, r^j \circ h \rbrace$.
        \end{enumerate}

    Case 2: $I$ have one element. Let $k \in \{ 1,2,3 \}$ be such that $I = \{ k \}$. For any $i,j \in \{ 1,2,3 \} \setminus I$, there exist $q_i, q_j, c_i, c_j \in \mathbb{Q}$ such that $(r^i \circ h)_n = q_i (r^k \circ h)_n + c_i$ and $(r^j \circ h)_n = q_j (r^k \circ h)_n + c_j$, for every $n \in \omega$. We claim that exists $d \in \mathbb{Q}_+$ such that $q_id + c_i \geq 0$, $q_jd + c_j \geq 0$ and $F = \lbrace (r^k \circ h) - d \rbrace$. In fact, we have the next subcases:  
        \begin{enumerate}
            \item[i)] The two sequences are constants. We obtain $q_i =0$, $q_j = 0$, $c_i \geq 0$ and $c_j \geq 0$. It works for $d = 0$.
            \item[ii)] Only one sequence is constant. Assume that $r^j \circ h$ is constant. Whenever $r^i \circ h$ have same or different kind of monotonicity of $r^k \circ h$, we either have $q_i > 0$, in which case we take $d \in \mathbb{Q}_+$ greater than $\frac{-c_i}{q_i}$, or $q_i < 0$ and $c_i > 0$, in which case we set $d=0$.
            \item[iii)] The two sequences are strictly monotone. If $r^i \circ h$ and $r^j \circ h$ have the same kind of monotonicity, we either have $q_i < 0$, $q_j < 0$, $c_i > 0$ and $c_j > 0$, in which case $d=0$ is enough, or $q_i > 0$ and $q_j > 0$, in which case we fix $d \in \mathbb{Q}_+$ such that $d > \max \left\lbrace \frac{-c_i}{q_i}, \frac{-c_j}{q_j} \right\rbrace$. Else, assume that $r^i \circ h$ is the only sequence which have the same kind of monotonicity of $r^k \circ h$. So $q_i > 0$, $q_j < 0$ and $c_j > 0$. Then, we obtain $\frac{c_i}{q_i} - \frac{c_j}{q_j} = \frac{(r^j \circ h)_n}{-q_j} + \frac{(r^i \circ h)_n}{q_i} > 0$, for every $n \in \omega$. Hence, we choose a positive rational number $d \in \left( \frac{-c_i}{q_i}, \frac{c_j}{-q_j} \right)$.  
        \end{enumerate}
    \end{proof}

    We will take advantage of Theorem \ref{homo} to define a group topology on $\mathbb Q ^{(\mathfrak c)}$ as follows.

    \begin{lem} \label{LemaW3}
    Fix an enumeration $\lbrace F_{\xi} \mid \omega \leq \xi < \mathfrak{c} \rbrace$ of $\mathcal H$ and let $\lbrace p_{\xi} \mid \omega \leq \xi < \mathfrak{c} \rbrace$ be a family of incomparable selective ultrafilters. Then exists a group topology on $\mathbb{Q}^{(\mathfrak{c})}$ such that, for every $\xi \in \mathfrak{c} \setminus \omega$, $\mathbb{Q}^{(\mathfrak{c} \setminus \lbrace 0 \rbrace)}$ is $p_{\xi}-$compact and $p_{\xi} - \lim\limits_{n \to \infty} \frac{1}{N}f_n \chi_0 \in \mathbb{Q}^{(\mathfrak{c} \setminus \lbrace 0 \rbrace)}$, for any $N \in \mathbb{N}$ and $f \in F_{\xi}$.   
    \end{lem}

    \begin{proof}
    Let $\lbrace A_{\xi} \mid \omega \leq \xi < \mathfrak{c} \rbrace \subset [ \mathfrak{c} \setminus \omega ]^{\mathfrak{c}}$ be a family of disjoint subsets. For every $\xi \in  \mathfrak{c} \setminus \omega$, consider an increasing enumeration $\lbrace \lambda_{\mu} \mid \mu < \mathfrak{c} \rbrace$ of $A_{\xi}$ and let $\lbrace h_{\alpha} \mid \alpha \in A_{\xi} \rbrace$ be a family of injective sequences of $\mathbb{Q}^{(\mathfrak{c})}$ such that
        \begin{itemize}
            \item $\displaystyle\bigcup_{n \in \omega} \supp{h_{\alpha}(n)} \subset \alpha$, for every $\alpha \in A_{\xi}$, 
            \item $\lbrace h_{\lambda_{\mu}} \mid \mu < \vert F_{\xi} \vert \rbrace$ is an indexation of $\lbrace f. \chi_0 \mid f \in F_{\xi} \rbrace$,
            \item $\lbrace [h_{\lambda_{\mu}}]_{p_{\xi}} \mid \vert F_{\xi} \vert \leq \mu < \mathfrak{c} \rbrace \cup \lbrace [ \chi_{\overrightarrow{\mu}} ]_{p_{\xi}} \mid \mu \in \mathfrak{c} \setminus \lbrace 0 \rbrace \rbrace$ is a basis of $\left( \mathbb{Q}^{(\mathfrak{c} \setminus \lbrace 0 \rbrace)} \right)^{\omega} \big/ p_{\xi}$.
        \end{itemize}
    As a consequence, $\lbrace [h_{\beta}]_{p_{\xi}} \mid \beta \in A_{\xi} \rbrace \cup \lbrace [ \chi_{\overrightarrow{\mu}} ]_{p_{\xi}} \mid \mu \in \mathfrak{c} \rbrace$ is linearly independent, for each $\xi \in  \mathfrak{c} \setminus \omega$. In virtue of Theorem \ref{homo}, for every $d \in \mathbb{Q}^{(\mathfrak{c})}$ non-zero, exists a homomorphism $\varphi_d : \mathbb{Q}^{(\mathfrak{c})} \to \mathbb{T}$ such that 
        \begin{enumerate}
            \item[a)] $\varphi_d(d) \neq 0$,
            \item[b)] $p_{\xi}- \lim\limits_{n \to \infty} \varphi_d \left( \frac{1}{N} h_{\beta}(n) \right) = \varphi_d \left( \frac{1}{N} \chi_{\beta} \right)$, for any $N \in \mathbb{N}$, $\beta \in A_{\xi}$ and $\xi \in [\omega, \mathfrak{c})$.
        \end{enumerate}
    Hence, the group topology on $\mathbb{Q}^{(\mathfrak{c})}$ induced by these homomorphisms is such that $p_{\xi}- \lim\limits_{n \to \infty}  \frac{1}{N} h_{\beta}(n) = \frac{1}{N} \chi_{\beta}$, for any $N \in \mathbb{N}$, $\beta \in A_{\xi}$ and $\xi \in [\omega, \mathfrak{c})$. Therefore, for every $\xi \in \mathfrak{c} \setminus \omega$, $\mathbb{Q}^{(\mathfrak{c} \setminus \lbrace 0 \rbrace)}$ is $p_{\xi}-$compact and $p_{\xi} - \lim\limits_{n \to \infty} \frac{1}{N}f_n \chi_0 \in \mathbb{Q}^{(\mathfrak{c} \setminus \lbrace 0 \rbrace)}$, for any $N \in \mathbb{N}$ and $f \in F_{\xi}$.  
    \end{proof}

    Finally, we provide a solution to the Problem 6.4 of \cite{Boero} for the particular case $\kappa = 3$. 

    \newpage

    \begin{teo}
    If exists $\mathfrak{c}$ incomparable selective ultrafilters, then there exists a Wallace semigroup $S$ such that $S^3$ is countably compact.
    \end{teo}

    \begin{proof}
    Let $\lbrace F_{\xi} \mid \omega \leq \xi < \mathfrak{c} \rbrace$, $\lbrace p_{\xi} \mid \omega \leq \xi < \mathfrak{c} \rbrace$ and $\tau$ be a group topology on $\mathbb{Q}^{(\mathfrak{c})}$ as in Lemma \ref{LemaW3}. We define $S = \lbrace r \chi_0 \mid r \in \mathbb{Q}_+ \rbrace \oplus \mathbb{Q}^{(\mathfrak{c} \setminus \lbrace 0 \rbrace)}$ and notice that $(S, \tau)$ is a Wallace semigroup. 

    For every sequence $(y_n)_{n \in \omega} \subset S^3$, there exists $r^1, r^2, r^3 \in \mathbb{Q}_+^{\omega}$ and $g^1, g^2, g^3 \in \left( \mathbb{Q}^{(\mathfrak{c} \setminus \lbrace 0 \rbrace)} \right)^{\omega}$ such that $y_n = (r^1_n \chi_0, r^2_n \chi_0, r^3_n \chi_0) + (g^1_n, g^2_n, g^3_n)$, for each $n \in \omega$. By Proposition \ref{PropW3}, exists $h : \omega \to \omega$ strictly increasing and $\xi \in \mathfrak{c} \setminus \omega$ such that
        \begin{eqnarray*}
            (r^1 \circ h)(n) \chi_0 & = & k_1 \chi_0 + \displaystyle\sum_{f \in F_{\xi}} a_f.f(n) \chi_0, \\
            (r^2 \circ h)(n) \chi_0 & = & k_2 \chi_0 + \displaystyle\sum_{f \in F_{\xi}} b_f.f(n) \chi_0, \\
            (r^3 \circ h)(n) \chi_0 & = &  k_3 \chi_0 + \displaystyle\sum_{f \in F_{\xi}} c_f.f(n) \chi_0,
        \end{eqnarray*}
    for every $n \in \omega$, where $k_1, k_2, k_3 \in \mathbb{Q}_+$ and $a_f, b_f, c_f \in \mathbb{Q}$ for any $f \in F_{\xi}$. We can consider $a_f = \frac{m_{a_f}}{n_{a_f}}$, $b_f = \frac{m_{b_f}}{n_{b_f}}$ and $c_f = \frac{m_{c_f}}{n_{c_f}}$, where $m_{a_f}, m_{b_f}, m_{c_f} \in \mathbb Z$ and $n_{a_f}, n_{b_f}, n_{c_f}  \in \mathbb N$, for any $f \in F_{\xi}$. 

    Then,  
        \begin{eqnarray*}
            p_{\xi}- \lim\limits_{n \to \infty} r^1_n \chi_0 & = & k_1 \chi_0 + \displaystyle\sum_{f \in F_{\xi}} m_{a_f}. \left( p_{\xi}- \lim\limits_{n \to \infty} \frac{1}{n_{a_f}} f(n) \chi_0 \right) \in S, \\
            p_{\xi}- \lim\limits_{n \to \infty} r^2_n \chi_0 & = & k_2 \chi_0 + \displaystyle\sum_{f \in F_{\xi}} m_{b_f}. \left( p_{\xi}- \lim\limits_{n \to \infty} \frac{1}{n_{b_f}}  f(n) \chi_0 \right) \in S, \\
            p_{\xi}- \lim\limits_{n \to \infty} r^3_n \chi_0 & = &  k_3 \chi_0 + \displaystyle\sum_{f \in F_{\xi}} m_{c_f} . \left( p_{\xi}- \lim\limits_{n \to \infty} \frac{1}{n_{c_f}} f(n) \chi_0 \right) \in S.
        \end{eqnarray*}
    Furthermore, we have $p_{\xi}- \lim\limits_{n \to \infty} g^i_n \in \mathbb{Q}^{(\mathfrak{c} \setminus \lbrace 0 \rbrace)}$, for every $1 \leq i \leq 3$. Therefore, $p_{\xi}- \lim\limits_{n \to \infty} y_n \in S^3$. Hence, $S^3$ is countably compact.
    \end{proof}

\section{Some remarks and questions}
	
    Here we point out some remarks and natural questions that are still open. We recall the following question from \cite{Boero}.   

    \begin{question}
    For which cardinal $\kappa \in (3, 2^{\mathfrak{c}}]$, there exists a Wallace semigroup whose powers smaller than $\kappa$ are countably compact?
    \end{question}
    
    For case $\kappa \leq \omega$, it would be interesting to find the largest number, if it exists, that admits more positive translations so that the Proposition \ref{PropW3} continues to work. Also, we already know that there are no Wallace semigroups $S$ such that  $S^{2^\mathfrak c}$ is countably compact (\cite{T96}) which gives the upper limitation for $\kappa$ in the question above.
    
    Moreover, this suggest to us adapting Comfort's question for Wallace semigroups.

    \begin{question}
    For every cardinal $\alpha \leq 2^{\mathfrak{c}}$, there exists a Wallace semigroup $S$ such that $S^{\gamma}$ is countably compact for all $\gamma < \alpha$, but $S^{\alpha}$ is not countably compact?
    \end{question}

	\bibliography{gruposcomfort}{}

\begin{thebibliography}{1}
	
	
    \bibitem{Tomita21} M. K. Bellini, V. O. Rodrigues and A. H. Tomita, {\em On countably compact group topologies without non-trivial convergent sequences on $\mathbb{Q}^{(\kappa)}$ for arbitrarily large $\kappa$ and a selective ultrafilter}, Topology and its Applications, 294 (2021), 107653.
    
    \bibitem{Boero} A. C. Boero and A. H. Tomita, {\em A group topology on the free abelian group of cardinality $\mathfrak{c}$ that makes its square countably compact}, Fundamenta Mathematicae, 3.212 (2011), 235-260.
    
    \bibitem{Comfort} W. W. Comfort, {\em Problems on topological groups and other homogeneous spaces: Open Problems in Topology}, North-Holland, 1990, 311-347.
    
    \bibitem{G} J. Ginsburg and V. Saks, {\em Some applications of ultrafilters in topology}, Pacific Journal of Mathematics, 57.2 (1975), 403-418.
    
    \bibitem{Grant} D. L. Grant, {\em Sequentially compact cancellative topological semigroups: some progress on the Wallace problem}, Annals of the New York Academy of Sciences, 704.1 (1993), 150-154.
    
    \bibitem{R} E. A. Reznichenko, {\em Extension of functions defined on products of pseudocompact spaces and continuity of the inverse in pseudocompact groups}, Topology and its Applications, 59.3 (1994), 233-244.
    
    \bibitem{RS} D. Robbie and S. Svetlichny, {\em An answer to A. D. Wallace’s question about countably compact cancellative semigroups}, Proceedings of the American Mathematical Society, 124.1 (1996), 325-330.
    
    \bibitem{Tomita05} A. H. Tomita, {\em A solution to Comfort's question on the countable compactness of powers of a topological group}, Fundamenta Mathematicae, 1.186 (2005), 1-24.
    
    \bibitem{T96} A. H. Tomita, {\em The Wallace Problem: a counterexample from $MA_{countable}$ and $p-$compactness}, Canadian Mathematical Bulletin, 39.4 (1996), 486-498.
    
    \bibitem{T19} A. H. Tomita, {\em A van Douwen-like ZFC theorem for small powers of countably compact groups without non-trivial convergent sequences}, Topology and its Applications, 259 (2019), 347-364.
    
    \bibitem{T98} A. H. Tomita, {\em The existence of initially $\omega_1$-compact group topologies on free Abelian groups is independent of ZFC}, Commentationes Mathematicae Universitatis Carolinae, 39.2 (1998), 401-413.
    
    \bibitem{T15} A. H. Tomita, \emph{A group topology on the free abelian group of cardinality ${\mathfrak c}$ that makes its finite powers countably compact}, Topology and its Applications, 196 (2015), 976-998.
    
    \bibitem{W} A. D. Wallace, {\em The structure of topological semigroups}, Bulletin of the American Mathematical Society, 61.2 (1955), 95-112.

\end{thebibliography}
	\bibliographystyle{amsplain}

\end{document}